\documentclass[a4paper,reqno]{amsart}

\usepackage{mathtools,cite,amsthm,enumerate}

\newcommand*{\abs}[1]{\left\lvert#1\right\rvert}
\newcommand*{\norm}[1]{\left\lVert#1\right\rVert}
\newcommand{\R}{\mathbb R}
\newcommand{\A}{\mathcal A}
\DeclareMathOperator{\sign}{sign}
\allowdisplaybreaks

\newtheorem{theorem}{Theorem}[section]
\newtheorem{lemma}{Lemma}[section]
\newtheorem{corollary}{Corollary}[section]
\newtheorem{proposition}{Proposition}[section]
\theoremstyle{definition}
\newtheorem{definition}{Definition}[section]

\numberwithin{equation}{section}

\begin{document}


\title[Fractional stochastic heat equation]{Fractional stochastic heat equation\\ with piecewise constant coefficients}

\author{Yuliya Mishura}

\address[Y. Mishura]{Department of Probability Theory, Statistics and Actuarial Mathematics, Taras Shevchenko National University of Kyiv, 64 Volodymyrska St.,
Kyiv, 01601,
Ukraine}
\email{myus@univ.kiev.ua}

\author{Kostiantyn Ralchenko}

\address[K. Ralchenko]{Department of Probability Theory, Statistics and Actuarial Mathematics, Taras Shevchenko National University of Kyiv, 64 Volodymyrska St.,
Kyiv, 01601,
Ukraine}
\email{k.ralchenko@gmail.com}

\author{Mounir Zili}

\address[M. Zili]{University of Monastir, Department of Mathematics, Faculty of sciences of Monastir,  Avenue de l'environnement, 5019 Monastir, Tunisia}
\email{mounir.zili@fsm.rnu.tn}

\author{Eya Zougar}

\address[E. Zougar]{University of Monastir, Department of Mathematics, Faculty of sciences of Monastir,  Avenue de l'environnement, 5019 Monastir, Tunisia}
\email{zougareya@gmail.com}

\begin{abstract}
We introduce  a  fractional stochastic heat equation with second order elliptic operator in divergence form,  having a piecewise constant diffusion coefficient, and driven by an infinite-dimensional fractional Brownian motion.  We characterize the fundamental solution of its deterministic part, and prove the existence and the uniqueness of its solution.
\end{abstract}

\keywords{Stochastic partial differential equation; discontinuity of coefficients; fundamental solution; mild solution; infinite-dimensional fractional Brownian motion.}

\subjclass[2010]{60G22, 60H15, 35R60.}

\maketitle

\section{Introduction}

In the last years, stochastic partial differential equations (SPDEs) driven by  different types of fractional noises have attracted a great attention of the probability community. In several cases, the noise is expressed as a function of a formal derivative of a one-dimensional fractional Brownian motion. 

A  one-dimensional fractional Brownian motion (fBm) $ (B^H(t))_{t\ge 0}$ with Hurst index $H \in (0,1)$  is a centered Gaussian process on some probability space $(\Omega , {\mathcal F}, {\mathbb P})$ with covariance function
\begin{equation}
\label{covfBm}
{\mathbb E}(B^H(t) B^H(s))= \frac{1}{2} \Big( t^{2H} + s^{2H} - | t-s |^{2H} \Big) .
\end{equation}
So, an fBm is a natural extension of a standard Brownian motion because taking $H= \frac{1}{2}$ in (\ref{covfBm}) we get
$${\mathbb E}(B^{1/2}(t) B^{1/2}(s))= s \wedge t .$$
Firstly introduced by Kolmogorov in \cite{Kolmogorov},  since the appearance of the paper \cite{Mandelbrot}, the interest in this process has increased enormously
as an important ingredient of fractal models because of such  characteristics as self-similarity, H\"older continuity and long-range dependence. Theory of stochastic calculus with respect to fBm has been developed, leading  to the consideration  of several types  of SPDEs driven by a noise depending, in one way or another, on a fBm (e.\,g.\ \cite{Nualart,Hu-Y,Marta2,Balan,BalanTudor,Jiang,Bo,c} and references therein).

In this paper, we consider the following SPDE:
\begin{equation}\label{e:2}
\begin{cases}
d u(t, x) = {\mathcal L} u(t, x) \,dt + h\bigl(u(t,x)\bigr)\,W^H (dt,x), 
& t\in (0,T], \; x\in \R ,\\
u(0,x) = 0, & x\in\R,
\end{cases}
\end{equation}
where $h$ is an affine function,
\begin{equation}\label{lin}
h(z) = h_1 z + h_2,
\end{equation}
$h_1, h_2\in\R$ and the operator  ${\mathcal L}$ is defined by
\[
 {\mathcal L} = \frac{1}{2\rho (x)} \frac{d}{dx} \left( \rho (x)A(x) \frac{d}{dx} \right).
\]
Here the coefficients $A$ and $\rho$ have the following form 
$$A(x)=  a_1 {\mathbf 1}_{\{ x \le 0 \} } + a_2 {\mathbf 1}_{\{ x > 0\}}  \quad \hbox{and} \quad
\rho(x)= \rho_1 {\mathbf 1}_{\{ x \le 0 \} } +\rho_2 {\mathbf 1}_{\{ x > 0\}},
$$
$a_i, \rho_i$ ($i=1, 2$) are  positive constants, and
$ \frac{df}{dx}$ denotes the derivative  of $f$ in the distributional sense. 
The term $W^H$  refers to an $L^2(\R)$-valued fractional Brownian motion with  Hurst index $H\in (\frac12,1)$, defined by 
\begin{equation}
\label{whl2}
   W^H(\cdot,t): = \sum_{j=1}^\infty \lambda_j\,e_j(\cdot)\,B^H_j(t),
\end{equation}
where $B^H_j=\{B^H_j(t), t\geq 0 \}$, $j\in \mathbb{N}$  is a sequence of one-dimensional fractional Brownian motions with the Hurst index $H\in (1/2, 1)$ starting at  the origin, $\{\lambda_j,j\in \mathbb{N} \}$ is a sequence of positive real numbers and $\{e_j,j\in \mathbb{N} \}$ is an orthonormal basis of $L^2(\R),$ such that:
\begin{equation}\label{ej}
  \sup_j \left\|e_j\right\|_\infty < \infty, \;\; \; \text{and} \;\;\; \sum_{j=1}^\infty \lambda_j < \infty ,
\end{equation}
where $\norm{\,\cdot\,}_\infty$ denotes the norm in $L^\infty(\R)$.
The series (\ref{whl2}) converges  a.\,s.\ in $L^2(\R)$ because of  the assumptions \eqref{ej}, see \cite{NU}.

Equation (\ref{e:2})  can be considered as a stochastic counterpart of the parabolic equation
\begin{equation}\label{e:3}
  \frac{\partial u(t, x)}{\partial t} = {\mathcal L} u(t, x),  
\end{equation}
which arises in mathematical  modelling of  diffusion phenomena in many areas, such as    ecology \cite{Can},   biology \cite{Nic} etc.
The non-smoothness of the coefficients reflects the heterogeneity of the medium in which the process under study propagates. An explicit expression of the fundamendal solution of equation (\ref{e:3}) was given in \cite{Zi,Zil,CZ,MZ0}. 

SPDEs with the operator ${\mathcal L}$ have been introduced and investigated in \cite{ZZ,ZZ2}, with an additive noise 
$W$ defined as a  centered Gaussian field
$ W = \{ W(t,C); t \in [0,T ], C \in B_b({\mathbb R}) \}$ with covariance
$${\mathbb E} (W(t,C)W(s,D)) = (t \wedge s) \lambda (C \cap D),$$
where
$ \lambda$  denotes the Lebesgue measure.
Comparing to those two articles, the equation (\ref{e:2}) contains more complicated noise, which is multiplicative and fractional.
Its study requires more sophisticated stochastic and Hilbert analysis tools.

Other forms of SPDE (\ref{e:2}), with other different operators and similar or different Gaussian noises, have been recently studied by many authors (e.\,g.\ \cite{Marta,Marta2,Tindel-Tudor,c} and references therein).

We make here a first step in the study of SPDEs of the form (\ref{e:2}). More precisely, we prove  existence and uniqueness of a mild solution to equation (\ref{e:2}). In addition to  the introduction of a Besov-type Banach space and the use of the Riemann--Liouville fractional derivatives,   a part of the theory of generalized Lebesgue--Stieljes integration with respect to fBm with Hurst index $H > 1/2$ is employed, and the proofs  require many  integration techniques, calculation and analysis tools; they are particularly  based on a deep characterization  of the explicit form of the fundamental solution of (\ref{e:3}). 
 
 The paper is organized as follows.
In the second section  we present some useful characterizations and upper bounds of the fundamental solution of equation  (\ref{e:3}), and   the third one  is devoted to the proof of the existence and uniqueness of a mild solution to equation (\ref{e:2}).
Some technical lemmata are proved in the appendix.

\section{Some properties of the fundamental solution}

The fundamental solution of the deterministic partial differential equation (\ref{e:3})
is given by
\begin{multline*}
G(t-s, x, y)
 =   \Bigg[ \frac{1}{\sqrt{2 \pi (t-s)}}
 \left( \frac{ {\bf 1}_{ \{ y \le 0 \} } }{\sqrt{a_1}} + \frac{ {\bf 1}_{ \{y>0\} } }{\sqrt{a_2}} \right)  \left\{ \exp \left( - \frac{(f(x) - f(y))^2}
{2 (t-s)}\right) \right. \\
  +  \left.
\beta\, {\rm sign} (y)\, \exp
\left( - \frac{(\mid f(x) \mid  + \mid f(y) \mid)^2}{2(t-s)}\right) \right\} \Bigg] {\bf 1}_{\{s < t\}},
\end{multline*}
with
$$f(y)= \frac{y}{\sqrt{a_1}}{\mathbf 1}_{\{y\leq 0\}}+\frac{y}{\sqrt{a_2}} {\mathbf 1}_{\{y>0\}},
\quad
\alpha  = 1- \frac{\rho_1 a_1}{\rho_2a_2},
\quad\text{and}\quad
\beta= \frac{\sqrt{a_1}+ \sqrt{a_2} (\alpha -1)}{\sqrt{a_1}-\sqrt{a_2} (\alpha -1)}   .$$
For the proof see, e.\,g., \cite{Zi,Zil} and \cite{CZ}.

 The fundamental solution $G$ satisfies the following  properties.

\begin{lemma}\label{lem1} For every $0\le s< t \le T$ and $x,y \in \R,$  we have
\[
\left|G(t-s, x, y)\right| \leq C_{a_1,a_2}\, \frac{1}{\sqrt{t-s}}\exp\left( - \frac{(f(x)-f(y))^2}{2(t-s)}\right),
\]
with $C_{a_1,a_2}:=\frac{ 1+\vert \beta \vert}{\sqrt{2 \pi }}\;\left( \frac{1}{\sqrt{a_1}}+\frac{1}{\sqrt{a_2}} \right) .$
\end{lemma}
\begin{proof}
 We have for every $x,y\in \R$ and $0\le s< t \le T$,
\begin{align*}
\left|G(t-s, x, y)\right| 
&=  \Biggl| \frac{1}{\sqrt{2 \pi (t-s)}}\left( \frac{ {\bf 1}_{ \{ y \le 0 \} } }{\sqrt{a_1}}+\frac{ {\bf 1}_{ \{ y > 0 \} } }{\sqrt{a_2}} \right)
\biggl\{ \exp \left( - \frac{(f(x) - f(y))^2}{2 (t-s)}\right)
\\
&\quad+ \beta\, \sign (y)\, \exp\left( - \frac{(\abs{f(x)} + \abs{f(y)})^2}{2(t-s)}\right) \biggr\}  \Biggr|
\\
&\leq \frac{1}{\sqrt{2 \pi (t-s)}}\left( \frac{1 }{\sqrt{a_1}}+\frac{ 1 }{\sqrt{a_2}} \right)
\biggl\{ \exp \left( - \frac{(f(x) - f(y))^2}{2 (t-s)}\right)\\
&\quad+
\left|\beta\right|\,  \exp\left( - \frac{(\abs{f(x)}  + \abs{f(y)})^2}{2(t-s)}\right) \biggr\}.
\end{align*}

Since the function $ x\mapsto \exp(-x^2)$ is decreasing on the interval $[0, + \infty )$, we have
\[\exp\left(- \frac{\left(\left|f(x)\right|+\left|f(y)\right|\right)^2}{2(t-s)}\right)\leq\, \exp\left(-\frac{(f(x)-f(y))^2}{2(t-s)} \right) , \]
and consequently
 \[
 \left|G(t-s, x, y)\right|\, \leq\, \frac{ (1+\vert \beta \vert)}{\sqrt{2 \pi (t-s)}}\;\left( \frac{1}{\sqrt{a_1}}+\frac{1}{\sqrt{a_2}} \right)\exp \left( - \frac{(f(x) - f(y))^2}{2 (t-s)}\right) .
\]
\end{proof}
\begin{corollary}\label{cor2}For every $0\le s< t \le T$ $x \in \R$ and $y\in \R,$ we have
\[
\max \Big( \int_{\R} \left|G(t-s, z, y)\right|\,dz, \int_{\R} \left|G(t-s, x, z)\right|\,dz  \Big) \leq\,C_1(a_1,a_2),
\]
with
$C_1(a_1,a_2)=\left(\frac{1}{\sqrt{a_1}}+\frac{1}{\sqrt{a_2}}\right)\,(1+\vert \beta \vert) \,\max(\sqrt{a_1},\sqrt{a_2}).$

\end{corollary}
\begin{proof}$:$
By the virtue of Lemma  \ref{lem1}, we get:
$$ \int_{\R} \left|G(t-s, z, y)\right|\,dz \leq   \frac{ (1+\vert \beta \vert)}{\sqrt{2 \pi (t-s)}}\;\left( \frac{1}{\sqrt{a_1}}+\frac{1}{\sqrt{a_2}} \right)\int_{\R}\exp \left( - \frac{(f(z) - f(y))^2}{2 (t-s)}\right)\;dz .$$
Further,
\begin{multline*}
\int_{\R}\exp \left( - \frac{(f(z) - f(y))^2}{2 (t-s)}\right)\;dz =\int_{-\infty}^0\exp \left( - \frac{(\frac{z}{\sqrt{a_1}} - f(y))^2}{2 (t-s)}\right)\,dz \\+ \int_{0}^{\infty}\exp \left( - \frac{(\frac{z}{\sqrt{a_2}} - f(y))^2}{2 (t-s)}\right)\,dz.
\end{multline*}
By the change of variables $\;U= \frac{\frac{z}{\sqrt{a_i}}-f(y)}{\sqrt{2(t-s)}}\Longrightarrow\;dU= \frac{dz}{\sqrt{2(t-s)\,a_i}}$, with $i= 1$ in the first integral and $i=2$ in the second one, we obtain:
\begin{align}
\int_{\R}\exp \left( - \frac{(f(z) - f(y))^2}{2 (t-s)}\right)\,dz 
&\leq \max(\sqrt{a_1},\sqrt{a_2})\,\sqrt{2(t-s)}\int_{\R}\exp(-U^2)\,dU
\notag
\\
&=  \max(\sqrt{a_1},\sqrt{a_2})\,\sqrt{2(t-s)\,\pi}.
\label{Majerror}
\end{align}
Thus,
\[
\int_{\R} \left|G(t-s, z, y)\right|\,dz \leq \,\left(\frac{1}{\sqrt{a_1}}+\frac{1}{\sqrt{a_2}}\right)\,(1+\vert \beta \vert) \,\max(\sqrt{a_1},\sqrt{a_2}) .
\]
Concerning the second integral, also by the virtue of Lemma  \ref{lem1}  and (\ref{Majerror}), we get:
\begin{align*}
\int_{\R} \left|G(t-s, x, z)\right|\,dz &\leq   \frac{ (1+\vert \beta \vert)}{\sqrt{2 \pi (t-s)}}\;\left( \frac{1}{\sqrt{a_1}}+\frac{1}{\sqrt{a_2}} \right)\int_{\R}\exp \left( - \frac{(f(x) - f(z))^2}{2 (t-s)}\right)\;dz
\\
&\leq\,\left(\frac{1}{\sqrt{a_1}}+\frac{1}{\sqrt{a_2}}\right)\,(1+\vert \beta \vert) \,\max(\sqrt{a_1},\sqrt{a_2}) .
\end{align*}
\end{proof}

\begin{lemma}\label{lem3}
\begin{enumerate}[(i)]
\item
For every $\eta > 0$, there exists  a strictly positive constant $C$ such that, for every $0  < t \le T$ and  $x, y \in \R$,
\[
 \max \Big( \int_{\R}\left|\partial_tG(t, z , y)\right|^\eta\,dz , \int_{\R}\left|\partial_tG(t, x , z)\right|^\eta\,dz \Big) \leq C t^{-\frac{3}{2}\eta+\frac12}.
\]

\item
For every $\eta > 0$, there exists  a strictly positive constant $C$ such that, for every $0 \le s < t \le T$ and  $y \in \R$,
\[
\int_{\R}\left|\frac{\partial^2}{\partial t \partial s}G(t-s, z , y)\right|^\eta\,dz
\leq C (t-s)^{-\frac{5}{2}\eta+\frac12}.
\]
\end{enumerate}
\end{lemma}
\begin{proof}
$(i)$
For every $x, y\in \R,$ we have
\begin{align*}
\left|\partial_tG(t, x, y)\right|
&\leq \frac{1}{\sqrt{2\pi}}\left(\frac{1}{\sqrt{a_1}}+\frac{1}{\sqrt{a_2}} \right)  \Bigg[\frac{1}{2}t^{-3/2}\bigg\{\exp \left( - \frac{(f(x) - f(y))^2}{2 t}\right)
\\
&\quad
+|\beta| \, \exp\left( - \frac{(\abs{f(x)}  + \abs{f(y)})^2}{2t}\right) \bigg\}
\\
&\quad+\frac{1}{2}t^{-1/2}\Big\{\frac{(f(x)-f(y))^2}{t^2}
\exp \left( - \frac{(f(x) - f(y))^2}{2 t}\right)
\\
&\quad+|\beta|\,\frac{(\left|f(x)\right|+\left|f(y)\right|)^2}{t^2}\exp \left( - \frac{(|f(x)|+|f(y)|)^2}{2 t}\right)\Big\} \Bigg]
\\
&\leq \frac{1}{\sqrt{2\pi}}\left(\frac{1}{\sqrt{a_1}}+\frac{1}{\sqrt{a_2}} \right)  
\Bigg[\frac{(1+\,|\beta|)}{2}\,t^{-3/2}\exp \left( - \frac{(f(x) - f(y))^2}{2 t}\right)
  \\
&\quad+\frac{1}{2}t^{-3/2}\Big\{\frac{(f(x)-f(y))^2}{t}\exp \left( - \frac{(f(x) - f(y))^2}{2t}\right)
\\
&\quad+|\beta|\frac{(\left|f(x)\right|+\left|f(y)\right|)^2}{t}\exp \left( - \frac{(|f(x)|+|f(y)|)^2}{2t}\right)\Big\} \Bigg]
\end{align*}
Then, by using the fact that for every $a, x, y , z >0,$  we have $(x+y + z )^a\leq C(x^a+y^a + z^a)$, we get
\begin{align*}
\left|\partial_tG(t, x, y)\right|^\eta
&\leq \frac{C}{(\sqrt{2\pi})^\eta}
\left(\frac{1}{\sqrt{a_1}}+\frac{1}{\sqrt{a_2}} \right)^\eta 
\Bigg[\frac{(1+|\beta|)^\eta}{2^\eta}t^{-\frac{3}{2}\eta}\exp \left( -\eta \frac{(f(x) - f(y))^2}{2t}\right)
\\
&\quad+ \frac{1}{2^\eta}t^{-\frac{3}{2}\eta}\Big\{\left(\frac{(f(x)-f(y))^2}{t}\right)^\eta\exp \left( - \eta \frac{(f(x) - f(y))^2}{2t}\right)
\\
&\quad+ |\beta|^\eta\frac{(\left|f(x)\right|+\left|f(y)\right|)^{2\eta}}{t^\eta}\exp \left( - \eta\frac{(|f(x)|+|f(y)|)^{2}}{2t}\right)\Big\} \Bigg].
\end{align*}
Thus,
\begin{align*}
\MoveEqLeft\int_{\R}\left|\partial_tG(t, z, y)\right|^\eta\,dz
\leq\,  \frac{ t^{-\frac{3}{2}\eta}  }{(2\sqrt{2\pi})^{\eta}}\left(\frac{1}{\sqrt{a_1}}+\frac{1}{\sqrt{a_2}} \right)^{\eta} 
\\*
&\quad\times
\Bigg[(1+\,|\beta|)^\eta\int_{\R}\exp \left( - \eta\frac{(f(z) - f(y))^2}{2t}\right)\,dz
  \\
  &\quad\quad+ \int_{\R}\left(\frac{(f(z)-f(y))^2}{t}\right)^\eta\exp \left( - \eta\frac{(f(z) - f(y))^2}{2t}\right)dz
  \\
&\quad\quad+ |\beta|^\eta\int_\R\left(\frac{(\left|f(z)\right|+\left|f(y)\right|)^2}{t}\right)^\eta\exp \left( -\eta \frac{(|f(z)|+|f(y)|)^2}{2t}\right) d z \Bigg].
\end{align*}

On the one hand, by (\ref{Majerror}) we have
$$\begin{array}{rcl}
\int_{\R}\exp \left( - \eta\frac{(f(z) - f(y))^2}{2t}\right)\,dz &\leq& \frac{1}{\sqrt{\eta}}\max(\sqrt{a_1},\sqrt{a_2})\sqrt{2\pi t} .
\end{array}$$
 On the other hand, we have
\begin{align*}
\MoveEqLeft
\int_{\R}\left(\frac{(f(z)-f(y))^2}{t}\right)^\eta\,\exp \left( - \eta\frac{(f(z) - f(y))^2}{2t}\right)\,dz
\\
&= \int_{0}^{+\infty}\left(\frac{(\frac{z}{\sqrt{a_2}}-f(y))^2}{t}\right)^\eta\,\exp \left( -\eta \frac{( \frac{z}{\sqrt{a_2}}- f(y))^2}{2t}\right)\,dz
\\
&\quad + \int_{-\infty}^0\left(\frac{(\frac{z}{\sqrt{a_1}}-f(y))^2}{t}\right)^\eta\,\exp \left( - \eta\frac{( \frac{z}{\sqrt{a_1}}- f(y))^2}{2t}\right)\,dz
\\
&\leq \int_\R\left(\frac{(\frac{z}{\sqrt{a_2}}-f(y))^2}{t}\right)^\eta\,\exp \left( - \eta\frac{( \frac{z}{\sqrt{a_2}}- f(y))^2}{2t}\right)\,dz
\\
&\quad+ \int_\R\left(\frac{(\frac{z}{\sqrt{a_1}}-f(y))^2}{t}\right)^\eta\,\exp \left( -\eta \frac{( \frac{z}{\sqrt{a_1}}- f(y))^2}{2t}\right)\,dz
\\
&\leq \left(\frac{2}{\eta}\right)^{\eta+\frac{1}{2}}\;\max(\sqrt{a_1},\sqrt{a_2})\,  \sqrt{t} \int_\R z^{2\eta}\,e^{-z^2}\,dz .
\end{align*}
It is clear that the last integral converges for every $\eta>0$. Therefore,  we get
\[
\int_{\R}\left(\frac{(f(z)-f(y))^2}{t}\right)^\eta\,\exp \left( - \eta\frac{(f(z) - f(y))^2}{2t}\right)\,dz \leq\;  C(\eta , a_1, a_2)\sqrt{t}.
\]

Concerning the last integral, we have
\begin{align*}
\MoveEqLeft
\int_\R \left(\frac{(\left|f(z)\right|+\left|f(y)\right|)^2}{t}\right)^\eta\exp \left( - \eta\frac{(|f(z)|+|f(y)|)^2}{2t}\right)\, dz
\\*
&= \int_{\{  z y\leq 0\}}\left(\frac{(f(z)-f(y))^2}{t}\right)^\eta\,\exp \left( - \eta \frac{(f(z) - f(y))^2}{2t}\right) \, dz
\\
&\quad+ \int_{\{  z y\geq 0\}}\left(\frac{(f(z)+f(y))^2}{t}\right)^\eta\,\exp \left( - \eta\frac{(f(z) +f(y))^2}{2t}\right) dz
\\
&\leq \int_{\R}\left(\frac{(f(z)-f(y))^2}{t}\right)^\eta\,\exp \left( - \eta\frac{(f(z) - f(y))^2}{2t}\right)\,dz
\\
&\quad+ \int_{\R}\left(\frac{(f(z)+f(y))^2}{t}\right)^\eta\,\exp \left( -\eta \frac{(f(z) +f(y))^2}{2t}\right)\,dz.
\end{align*}
By the same technique as above, we get
\begin{multline*}
\int_\R\left(\frac{(\left|f(z)\right|+\left|f(y)\right|)^2}{t}\right)^\eta\exp \left( -\eta \frac{(|f(z)|+|f(y)|)^2}{2t}\right)\, dz
\\*
 \leq\left(\frac{2}{\eta}\right)^{\eta+\frac{1}{2}}\max(\sqrt{a_1},\sqrt{a_2})\, \sqrt{\frac{\pi t}{2}}.
\end{multline*}
All this implies that
\[
 \int_{\R}\left|\partial_tG(t, z, y)\right|^\eta\,dz \leq\, C(\eta , a_1, a_2, \beta) t^{-\frac{3}{2}\eta+\frac12},
\]
where $C(\eta , a_1, a_2,\beta)$ denotes a strictly positive constant depending only on $\eta , a_1, a_2, \beta $.

Following the same steps, we can prove that
\[
 \int_{\R}\left|\partial_tG(t, x, z)\right|^\eta\,dz \leq\, C(T, \eta , a_1, a_2, \beta) t^{-\frac{3}{2}\eta+\frac12}
\]
with the same constant $C(\eta , a_1, a_2,\beta)$ as above.

$(ii)$
The mixed partial derivative of $G$ equals
\begin{align*}
\MoveEqLeft
\frac{\partial^2}{\partial t \partial s}G(t-s, z , y)
= \frac{1}{\sqrt{2\pi}} \left( \frac{ {\bf 1}_{ \{ y \le 0 \} } }{\sqrt{a_1}} + \frac{ {\bf 1}_{ \{y>0\} } }{\sqrt{a_2}} \right)
\biggl[
\frac34(t-s)^{-\frac52} \exp\left\{-\frac{(f(z)-f(y))^2}{2(t-s)}\right\}
\\
&-3(t-s)^{-\frac72} \exp\left\{-\frac{(f(z)-f(y))^2}{2(t-s)}\right\} \frac{(f(z)-f(y))^2}{2}
\\
&+(t-s)^{-\frac92} \exp\left\{-\frac{(f(z)-f(y))^2}{2(t-s)}\right\} \frac{(f(z)-f(y))^4}{4}
\\
&+\frac34 \beta \sign(y) (t-s)^{-\frac52} \exp\left\{-\frac{(\abs{f(z)}+\abs{f(y)})^2}{2(t-s)}\right\}
\\
&-3\beta \sign(y)(t-s)^{-\frac72} \exp\left\{-\frac{(\abs{f(z)}+\abs{f(y)})^2}{2(t-s)}\right\} \frac{(\abs{f(z)}+\abs{f(y)})^2}{2}
\\
&+\beta \sign(y)(t-s)^{-\frac92} \exp\left\{-\frac{(\abs{f(z)}+\abs{f(y)})^2}{2(t-s)}\right\} \frac{(\abs{f(z)}+\abs{f(y)})^4}{4}
\biggr].
\end{align*}
Therefore
\begin{align*}
\MoveEqLeft
\abs{\frac{\partial^2}{\partial t \partial s}G(t-s, z , y)}
\le C (t-s)^{-\frac52}
\biggl[
\exp\left\{-\frac{(f(z)-f(y))^2}{2(t-s)}\right\}
\\
&+ \exp\left\{-\frac{(f(z)-f(y))^2}{2(t-s)}\right\} \frac{(f(z)-f(y))^2}{2(t-s)}
\\
&+\exp\left\{-\frac{(f(z)-f(y))^2}{2(t-s)}\right\} \frac{(f(z)-f(y))^4}{4(t-s)^2}
\\
&+\exp\left\{-\frac{(\abs{f(z)}+\abs{f(y)})^2}{2(t-s)}\right\}
+\exp\left\{-\frac{(\abs{f(z)}+\abs{f(y)})^2}{2(t-s)}\right\} \frac{(\abs{f(z)}+\abs{f(y)})^2}{2(t-s)}
\\
&+\exp\left\{-\frac{(\abs{f(z)}+\abs{f(y)})^2}{2(t-s)}\right\} \frac{(\abs{f(z)}+\abs{f(y)})^4}{4(t-s)^2}
\biggr].
\end{align*}
The rest of proof can be done similarly to that of $(i)$.
\end{proof}

\begin{corollary}\label{cor4}For all  $\delta \in (\frac13,1)$
there exists a strictly positive constant $C>0,$ such that, for every $x, y \in {\mathbb R}, $  for every
 $ s, t \in (0,T]\;$ with\; $s < t,$
\[
\max \Big( \int_\R\Big|G(t, z, y)-G(s, z, y)\,\Big|dz , \int_\R\Big|G(t, x, z)-G(s, x, z)\,\Big|dz \Big)
\leq C\, s^{-\delta}\,(t-s)^\delta.
\]
\end{corollary}
\begin{proof}
For every fixed $z, y\in\R$, $0< s < t \le T$ and $\delta\in[0,1]$  we have, by the triangular inequality,
\begin{align*}
\big|G(t, z, y)-G(s, z, y)\,\big|
&= \big|G(t, z, y)-G(s, z, y)\,\big|^{1-\delta} \;\big|G(t, z, y)-G(s, z, y)\,\big|^\delta
\\
&\leq \left(\big|G(t, z, y)\big|+ \big|G(s, z, y)\,\big|\right)^{1-\delta} \;\big|G(t, z, y)-G(s, z, y)\,\big|^\delta.
\end{align*} 
By mean-value theorem, there exists $t^*\in(s,t)$ such that
\[
\big|G(t, z , y)-G(s, z, y)\,\big|^\delta\leq\,(t-s)^\delta \; |\partial_tG(t^*,z, y) |^\delta.
\]
Therefore, by Lemma  \ref{lem1}  and since for every $x\in\R,\;e^{-x^2}\leq1,$ we have
 \[
\big|G(t, z, y)-G(s, z, y)\,\big|
\leq\,C\, \left(t^{-\frac{1}{2}}+ s^{-\frac{1}{2}}\right)^{1-\delta} \,
(t-s)^\delta \; |\partial_tG(t^*,z , y) |^\delta.
\]
Since for every $s < t,$ we have $t^{-1/2} < s^{-1/2},$ by Lemma  \ref{lem3},  we get:
\begin{align*}
\int_\R \big|G(t, z, y)-G(s, z, y)\,\big|\,dz
&\leq C s^{-\frac{1}{2}(1-\delta)}(t-s)^\delta \; \int_\R|\partial_tG(t^*, z,y) |^\delta\,dz
\\
&\leq C s^{-\frac{1}{2}(1-\delta)}(t-s)^\delta \;(t^*)^{-\frac{3}{2}\,\delta+\frac12}.
\end{align*}
Note that for $\delta>\frac13$
\[(t^*)^{-\frac{3}{2}\,\delta+\frac12} \le s^{-\frac{3}{2}\,\delta+\frac12}.\]
Consequently, we obtain:
\[
\int_\R \big|G(t, z, y)-G(s, z , y)\,\big|\,dz
\leq  C  s^{-\frac{1}{2}(1-\delta)}\, s^{-\frac{3}{2}\,\delta+\frac12}(t-s)^\delta 
\leq C s^{-\delta}\,(t-s)^\delta .
\]

Following the same technique we also get
\[
\int_\R \big|G(t, x, z)-G(s, x , z)\,\big|\,dz \le
 C\, s^{-\delta}\,(t-s)^\delta .
 \qedhere
 \]
\end{proof}

\begin{lemma}\label{lem6} For every  $\delta \in (\frac15,1)$  there exists a strictly positive constant $C$ such that, for all $0< \sigma <\tau <s <t <T,
$
\begin{multline}\label{eq:lem6}
 \int_\R \left|G(t- \tau, z, y)-G(s-\tau,z, y)-G(t- \sigma, z , y)+G(s- \sigma, z, y)\right|\,dz\\
 \leq\, C\,(t-s)^{\delta}\,(s-\tau)^{-2\delta}\,(\tau-\sigma)^{\delta}.
\end{multline}
\end{lemma}
\begin{proof}
On the one hand, by Lemma~\ref{lem1},
\begin{align*}
\MoveEqLeft
\left|G(t- \tau, z, y)-G(s-\tau,z, y)-G(t- \sigma, z , y)+G(s- \sigma, z, y)\right|
\\
&\le C \left( (t- \tau)^{-\frac12} + (s- \tau)^{-\frac12} + (t- \sigma)^{-\frac12} +(s- \sigma)^{-\frac12} \right)
\le C (s- \tau)^{-\frac12}.
\end{align*}
On the other hand, by the mean value theorem, there exist
$\theta^*\in(s,t)$ and $\rho^*\in(\sigma,\tau)$
such that
\begin{multline*}
\abs{G(t- \tau, z, y)-G(s-\tau,z, y)-G(t- \sigma, z , y)+G(s- \sigma, z, y)}
\\*
= \abs{\int_s^t\!\!\int_\sigma^\tau \frac{\partial^2}{\partial\theta \partial \rho} G(\theta-\rho,z,y)\,d\rho\,d\theta}
= \abs{\frac{\partial^2}{\partial\theta \partial \rho} G(\theta^*-\rho^*,z,y)} (t-s) (\tau-\sigma).
\end{multline*}
Hence, for every $\delta\in(0,1),$ we have
\begin{align*}
\MoveEqLeft
\int_\R \left|G(t- \tau, z, y)-G(s-\tau,z , y)-G(t- \sigma, z, y)+G(s- \sigma, z , y)\right|\,dz
\\
&=\int_\R \left|G(t- \tau, z , y)-G(s-\tau, z , y)-G(t- \sigma, z, y)+ G(s- \sigma, z , y)\right|^{1-\delta}
\\
&\quad\times \left|G(t- \tau, z , y)-G(s-\tau, z , y) -G(t- \sigma, z , y)+ G(s- \sigma, z , y)\right|^{\delta}\,dz
\\
&\le C(s-\tau)^{-\frac12(1-\delta)}(t-s)^\delta (\tau-\sigma)^\delta
\int_\R \abs{\frac{\partial^2}{\partial\theta \partial \rho} G(\theta^*-\rho^*,z,y)}^\delta \,dz.
\end{align*}
Taking into account Lemma~\ref{lem3} $(ii)$, we obtain
\begin{align*}
\MoveEqLeft
\int_\R \left|G(t- \tau, z, y)-G(s-\tau,z , y)-G(t- \sigma, z, y)+G(s- \sigma, z , y)\right|\,dz
\\
&\le C(s-\tau)^{-\frac12(1-\delta)}(t-s)^\delta (\tau-\sigma)^\delta
(\theta^*-\rho^*)^{-\frac52\delta+\frac12}.
\end{align*}
If $\delta>\frac15$, then
$(\theta^*-\rho^*)^{-\frac52\delta+\frac12} \le (s-\tau)^{-\frac52\delta+\frac12}$
and we arrive at \eqref{eq:lem6}.
\end{proof}

\section{Mild solution}
\subsection{Definitions and notation}
 \subsubsection{Norms and spaces}
Throughout the paper, the symbol $C$ will denote a generic constant,
the precise value of which is not important and may vary between different
equations and inequalities.
Let $\norm{\,\cdot\,}_2$  be the norm in $L^2(\R)$.
Let $0< \sigma < 1$. For every measurable function
$u\colon [0,T]\times\R\to \R$ and $t\in[0,T]$ denote
\begin{gather}
\norm{u}_{\sigma,1,t}^2:=   \int_0^t\left(\int_0^s\frac{\left\|u(s,\cdot)-u(v,\cdot)\right\|_2  }{(s-v)^{\sigma+1}}dv\right)^2 ds,
\notag\\
\label{lem7}
\left\| u \right\|^2_{\sigma,2,t}:=\sup_{s\in[0,t]}\left\| u(s,\cdot) \right\|_2^2+\norm{u}_{\sigma,1,t}^2.
\end{gather}
We define also the following seminorm for $f\colon[0,T]\to \R$ and $t\in[0,T]$: 
\[
\left\| f\right\|_{\sigma,0,t}:= \sup_{0 \leq u<v\leq t}\left( \frac{\left|f(u)-f(v)\right|}{(v-u)^{1-\sigma}}+\int_u^v \frac{\left|f(u)-f(z)\right|}{(z-u)^{2-\sigma}}\,dz      \right)
\]
For $\sigma \in(0,1)$ we denote by $ \mathcal{B}^{\sigma,2}\left([0,T];L^2(\R) \right)$ the following Banach space:
\begin{multline*}
\mathcal{B}^{\sigma,2}\left([0,T];L^2(\R) \right)
\coloneqq\Bigl\{
   u\colon[0,T]\times\R\to \R\text{ Lebesgue measurable mapping}\\ \text{such that} \; \left\| u \right\|^2_{\sigma,2,T}< \infty\Bigr\}.
\end{multline*}

\subsubsection{Integration with respect to $B^H$ for every $H\in (\frac{1}{2},1)$ }

Let $a,b\in \R,a<b.$ Let $\varphi\in L^1([a,b])$ and $\sigma \in(0,1).$ The Riemann--Liouville left- and right-sided fractional integrals of $\varphi$ of order $\sigma$ are defined for almost all $x\in[a,b]$ by
\begin{align*}
I_{a^+}^\sigma\varphi(x)&=  \frac{1}{\Gamma(\sigma)} \int_a^x (x-y)^{\sigma-1}\,\varphi(y)\,dy
\\
I_{b^-}^\sigma\varphi(x)&= \frac{1}{\Gamma(\sigma)} \int^b_x (x-y)^{\sigma-1}\,\varphi(y)\,dy.
\end{align*}

Let $I_{a^+}^\sigma(L^p) (resp.\;\;I_{b^-}^\sigma(L^p) ) $ denote the image of $L^p([a,b])$ by the operator $I_{a^+}^\sigma\, (resp.\;\,I_{b^-}^\sigma ).$ \\ If $\varphi \in I_{a^+}^\sigma(L^p) (resp.\;\;\varphi \in I_{b^-}^\sigma(L^p) )$ then the Riemann--Liouville left- and right-sided fractional derivatives are defined by
\begin{align*}
D^\sigma_{a^+}\varphi(x)&=\frac{1}{\Gamma(1-\sigma)}\left( \frac{\varphi(x)}{(x-a)^\sigma}+\sigma\int_a^x \frac{\varphi(x)-\varphi(y)}{(x-y)^{1+\sigma}}\,dy\right),
\\
D^{\sigma}_{b^-}\varphi(x)&=\frac{1}{\Gamma(1-\sigma)}\left( \frac{\varphi(x)}{(b-x)^{\sigma}}+\sigma\int_x^b \frac{\varphi(x)-\varphi(y)}{(y-x)^{1+\sigma}}\,dy\right).
\end{align*}

 For every two functions $\varphi,\psi\colon[a,b]\rightarrow \R$  such that $D^\sigma_{a^+}\varphi \in L^1([a,b])$ and
$D^{1-\sigma}_{b^-}\psi_{b^-} \in L^\infty([a,b]) $ where $\psi_{b^-}(x)=\psi(b^-)-\psi(x),$ we define the generalized Lebesgue--Stieljes integral by
\begin{equation}\label{stie}
\int^b_a\varphi(x)\,d\psi(x):= \int^b_aD^\sigma_{a^+}\varphi(x)\,D^{1-\sigma}_{b^-}\psi_{b^-}(x)\,dx.
\end{equation}
Furthermore, this integral admits the following bound:
\begin{equation}\label{majstie}
\left|\int^b_a\varphi(x)\,d\psi(x)\right|\leq\,C_{\sigma}\, \left\|\psi\right\|_{\sigma,0,b}
\,\int_a^b\left( \frac{\left|\varphi(x)\right|}{(x-a)^{\sigma}}+\int_a^x \frac{\left|\varphi(x)-\varphi(y)\right|}{(x-y)^{1+\sigma}}\,dy\right)dx.
\end{equation}

For more information on generalized Lebesgue--Stieljes integration with respect to fractional Brownian motion with Hurst index $H > 1/2$, the reader can see, e.\,g., \cite{Mishura} and references therein.

Now we suppose that  $\sigma \in (1-H,1/2)$, and denote by $\mathcal{L}\left(L^2({\mathbb R}) \right)$ the space of linear operators on $L^2({\mathbb R})$. Let $F\colon[0,T]\times\Omega \to \mathcal{L}\left(L^2({\mathbb R}) \right)$  be an operator function such that
\begin{equation}\label{eq:cond-int}
\sup_{j\in \mathbb{N}}\int_0^T \left(\frac{\left\|F(s,\omega)e_j \right\|_2}{s^\sigma}+\int_0^t \frac{\left\|(F(s,\omega)-F(v,\omega))e_j\right\|_2}{(s-v)^{\sigma+1}} \,dv\right)\,ds <\infty \qquad\text{a.\,s.}
\end{equation}
Following \cite{NU},  we define the integral with respect to  the $L^2({\mathbb R})$-valued fractional Brownian motion by
\begin{equation}\label{integ}
\int_a^b F(s,\omega)\,dB^H(s)=\sum_{j=1}^{\infty} \lambda_j
\int_a^b F(s,\omega)e_j\,dB_j^H(s),
\end{equation}
where the integral with respect to $B^H_j$ is the path-wise generalized Lebesgue--Stieltjes integral defined in (\ref{stie}) and the convergence of the series should be understood as $\mathbb{P}$-a.\,s.\ convergence in $L^2({\mathbb R})$. 

Applying assumptions \eqref{ej} and \cite[Proposition 2.1]{NU}, we deduce that the integral \eqref{integ} is well defined.

From (\ref{majstie}), one can obtain the following inequality
for $0\leq a<b\leq T$:
\begin{multline}\label{ineg}
\left\|\int_a^b F(s,\omega)\,dB^H(s)\right\|_2\;
\leq C \xi_{\sigma,H,T}\sup_{i\in \mathbb{N}}\int_a^b\biggl( \frac{\left\|F(s,\omega)e_i\right\|_2}{(s-a)^{\sigma}}
\\*
+\int^s_a \frac{\left\|F(s,\omega)e_i-F(v,\omega)e_i\right\|_2}{(s-v)^{1+\sigma}}\,dv\biggr)ds,
\end{multline}
where $C=C(\sigma)$ is a constant, and the random variable 
\[
\xi_{\sigma,H,T}
\coloneqq \sum_{j=1}^\infty\lambda_j\left\|B^H_j\right\|_{\sigma,0,T}
\]
is finite a.\,s., see \cite{NU}.
 For more details on the integration with respect to the $L^2({\mathbb R})$-valued fractional Brownian motion,  see e.\,g.\  \cite{NU} and references therein.

\subsubsection{Mild solution}
\begin{definition}
\label{MildDef}
 An $L^2(\R)$-valued random process $\{u(t,\cdot),t\in[0,T]\}$ is called a mild solution to the problem \eqref{e:2}  if it satisfies the following assumptions:
 \begin{enumerate}
 \item $u \in  \mathcal{B}^{\sigma,2}\left([0,T];L^2(\R) \right)$ a.s. for some $\sigma \in (1-H, 1/2)$.
 \item For any $\forall t \in [0,T]$ it holds that
\begin{equation}\label{sm}
u(t,x) =  \sum_{j=1}^\infty\lambda_j\int_0^t\!\!\int_\R G(t-s, x , y)\,h(u(s,y))\;e_j(y)\,dy\,dB_j^H(s) \qquad a.s.
\end{equation}
\end{enumerate}
Here, the integrals with  respect to $B_j^H, j \in {\mathbb N}$ are the generalized Lebesgue--Stieltjes integrals defined in (\ref{stie}).
\end{definition}

Let us introduce some notations for $u \in  \mathcal{B}^{\sigma,2}\left([0,T];L^2(\R) \right)$. Namely, let
\begin{gather}
\varsigma_{j,t}(u)(s,x)\coloneqq \int_\R G(t-s, x, y)\,h(u(s,y))\;e_j(y)\,dy,
\label{eq:zeta}
\intertext{and}
\varsigma^*_{j,t,s}(u)(v,x)\coloneqq\varsigma_{j,t}(u)(v,x)-\varsigma_{j,s}(u)(v,x).
\label{eq:zeta*}
\end{gather}
Also let us denote the right-hand side of \eqref{sm} by
\begin{align*}
(\A u)(t,x) &\coloneqq  \sum_{j=1}^\infty\lambda_j\int_0^t\int_\R G(t-s, x , y)\,h(u(s,y))\;e_j(y)\,dy\,dB_j^H(s)
\\
&=\sum_{j=1}^\infty\lambda_j\int_0^t\varsigma_{j,t}(u)(s,x)\,dB_j^H(s).
\end{align*}

Note that  $\A u$ is an integral with respect to an $L^2(\R)$-valued fractional Brownian motion, defined by \eqref{integ}.
The condition \eqref{eq:cond-int} for it has the form
\[
\sup_{j\in \mathbb{N}} \int_0^T \left(\frac{ \left\|\varsigma_{j,t}(u)(s,\cdot)\right\|_2 }{s^\sigma} + \int_0^s  \frac{    \left\| \varsigma_{j,t}(u)(s,\cdot)- \varsigma_{j,t}(u)(v,\cdot)\right\|_2}{(s-v)^{\sigma+1}}\,dv\right)ds
<\infty \qquad \text{a.\,s.}
\]
In fact, this condition holds for any $u \in  \mathcal{B}^{\sigma,2}\left([0,T];L^2(\R) \right)$.
It will be checked in the proof of Proposition~\ref{prop11} below,
in which we  also establish that  $\A u$ actually determines
an $L^2(\R)$-valued stochastic process from the class $\mathcal{B}^{\sigma,2}\left([0,T];L^2(\R) \right)$.

\subsection{A priori estimates}
Let us fix $H\in \left(\frac{1}{2},1\right)$ and $\sigma \in \left(1-H,\frac{1}{2}\right)$. Abbreviate $\xi = \xi_{\sigma,H,T}$.
\begin{proposition}\label{prop11}
Let $u \in  \mathcal{B}^{\sigma,2}\left([0,T];L^2(\R) \right)$.
Then $\A u \in  \mathcal{B}^{\sigma,2}\left([0,T];L^2(\R) \right)$ a.\,s.
Moreover, for any $t\in(0,T]$,
\[
\norm{\A u}_{\sigma,2,t}^2 \le C \xi^2 
\left(\norm{u}_{\sigma,2,t}^2+1\right).
\]
\end{proposition}

\begin{proof} 
Let $\delta\in(\max\{2\sigma,\frac13\},1)$ be fixed throughout the proof.
It follows from $(\ref{sm})$ and  Lemma \ref{lem8} that for any $t\in(0,T]$, 
\begin{align*}
\left\|(\A u)(t,\cdot)\right\|_2&= \left\| \sum_{j=1}^\infty \lambda_j\int_0^t \varsigma_{j,t}(u)(s,\cdot)dB_j^H(s)\right\|_2
\\
&\leq C\,\xi \sup_{j\in \mathbb{N}}\int_0^t\! \left(\frac{ \left\|\varsigma_{j,t}(u)(s,\cdot)\right\|_2 }{s^\sigma} + \!\int_0^s  \frac{    \norm{\varsigma_{j,t}(u)(s,\cdot)- \varsigma_{j,t}(u)(v,\cdot)}_2}{(s-v)^{\sigma+1}}dv\right)ds
\\ 
&\leq C\,\xi \Bigg(\int_0^t \frac{\left\|u(s,\cdot)\right\|_2+1 }{s^\sigma}\,ds+ \int_0^t\!\!\int_0^s
\frac{\left\|u(s,\cdot)-u(v,\cdot)\right\|_2  }{(s-v)^{\sigma+1}}\,dv\,ds
\\
&\quad + \int_0^t(t-s)^{-\frac{\delta}{2}}\int_0^s (s-v)^{\frac{\delta}{2}-\sigma-1} \left( \left\|u(v,\cdot)\right\|_2+1 \right)dv\,ds\Bigg).
\end{align*}
By the Cauchy--Schwarz inequality, we have
\[
\int_0^t \frac{\left\|u(s,\cdot)\right\|_2+1}{s^\sigma}\,ds \leq\,\left( \int_0^t \frac{1}{s^{2\sigma}}\,ds \right)^{1/2}\; \left(\int_0^t \left (\left\|u(s,\cdot)\right\|_2+1\right )^2 \;ds\right)^{1/2}.
\]
Since $\sigma < 1/2,$ we have $\int_0^t \frac{1}{s^{2\sigma}}\,ds < \infty$.
 Hence,
\begin{equation}\label{majo1}
\int_0^t \frac{\left\|u(s,\cdot)\right\|_2+1}{s^\sigma}\,ds
\leq\,C\left(\int_0^t \left (\left\|u(s,\cdot)\right\|_2^2+1\right ) \;ds\right)^{1/2}.
\end{equation}

By Fubini's Theorem and applying the change of variables $x=\frac{s-v}{t-v}$ in the last integral, we obtain
\begin{equation}\label{beta}
\begin{split}
\MoveEqLeft \int_0^t(t-s)^{-\frac{\delta}{2}}\int_0^s (s-v)^{\frac{\delta}{2}-\sigma-1}\, \left(\left\|u(v,\cdot)\right\|_2+1\right ) dv\,ds
\\
&= \int_0^t\left(\int_v^t(t-s)^{-\frac{\delta}{2}} 
(s-v)^{\frac{\delta}{2}-\sigma-1}ds\right) \left(\left\|u(v,\cdot)\right\|_2+1\right )\,dv
\\ 
&= \int_0^t(t-v)^{-\sigma}\left(\int_0^1(1-x)^{-\frac{\delta}{2}}x^{\frac{\delta}{2}-\sigma-1}\,dx \right) \left(\left\|u(v,\cdot)\right\|_2+1\right ) dv
\\
&= B\left(\tfrac{\delta}{2}-\sigma,1 -\tfrac{\delta}{2}\right)\,\int_0^t(t-v)^{-\sigma}\left(\left\|u(v,\cdot)\right\|_2+1\right ) dv
\\
&\leq C \left(\sup_{v \in [0,t]} \left\|u(v,\cdot)\right\|_2+1\right ) \int_0^t (t-v)^{-\sigma} dv
\\
&\le C t^{1 - \sigma }\left(\sup_{v \in [0,t]} \left\|u(v,\cdot)\right\|_2+1\right ), 
\end{split}
\end{equation}
where $B$ denotes the beta function defined, for every $p,q >0$, by $  B(p,q):= \int_0^1 t^{p-1}\, (1-t)^{q-1}\,dt$, and in the last inequality we used the fact that $\sigma < 1/2$.

Therefore, using $(\ref{majo1})$ and $(\ref{beta})$, we obtain  
\begin{align*}
\left\|(\A u)(t,\cdot)\right\|_2^2 &\leq\,C\,\xi^2\,\Bigg[ t^{1 - \sigma } \left(\sup_{v \in [0,t]} \left\|u(v,\cdot)\right\|_2+1\right )  + \left(\int_0^t \left (\left\|u(s,\cdot)\right\|_2^2+1\right ) \;ds\right)^{1/2}
\\*
&\quad+\int_0^t\int_0^s
\frac{\left\|u(s,\cdot)-u(v,\cdot)\right\|_2  }{(s-v)^{\sigma+1}}\,dv\,ds\Bigg]^2 \\
&\le  C\,\xi^2\, \Bigg(  t^{2 - 2\sigma } \left(\sup_{v \in [0,t]} \left\|u(v,\cdot)\right\|_2^2+1\right ) +  \int_0^t \left (\left\|u(s,\cdot)\right\|_2^2+1\right )\, ds
\\*
&\quad+\left( \int_0^t\int_0^s\frac{\left\|u(s,\cdot)-u(v,\cdot)\right\|_2  }{(s-v)^{\sigma+1}}\,dv\,ds\right)^2 \Bigg)
\\
&\leq C\,\xi^2\, \Bigg( (t^{2 - 2\sigma }+t) 
\left (\sup_{v \in [0,t]} \left\| u(v, \cdot ) \right\|_2^2+1\right)
\\
&\quad+ t \int_0^t \left( \int_0^s\frac{\left\|u(s,\cdot)-u(v,\cdot)\right\|_2  }{(s-v)^{\sigma+1}}\,dv \right)^2 \,ds \Bigg) .
\end{align*}
Since $t^{2 - 2\sigma }=t^{1 - 2\sigma }t \le T^{1 - 2\sigma }t = Ct$ for $\sigma<1/2$, we arrive at
\begin{equation}
\label{Ih2}
\left\|(\A u)(t,\cdot)\right\|_2^2
\leq C\,\xi^2\, t \Bigg( 
\sup_{v \in [0,t]} \left\| u(v, \cdot ) \right\|_2^2+1
+ \int_0^t \left( \int_0^s\frac{\left\|u(s,\cdot)-u(v,\cdot)\right\|_2  }{(s-v)^{\sigma+1}}\,dv \right)^2 \,ds \Bigg) .
\end{equation}
By the definition of the norm $\left\|.\right\|_{\sigma,2,t}$, we get 
\[
\norm{(\A u)(t,\cdot)}_2^2
\leq C\,\xi^2\, t \left(\norm{u}_{\sigma,2,t}^2+1\right).
\]
Consequently,
\begin{equation}\label{eq:Au-1}
\sup_{s\in[0,t]}\norm{(\A u)(s,\cdot)}_2^2
\leq C\,\xi^2\,\left(\norm{u}_{\sigma,2,t}^2+1\right),
\end{equation}
because $\left\| u\right\|^2_{\sigma,2,s}\leq \left\| u\right\|^2_{\sigma,2,t}$ for $s\leq t$.

\medskip

Futhermore, 
\begin{align}
\MoveEqLeft[1.5]\left\|(\A u)(s,\cdot) - (\A u)(v,\cdot)\right\|_2
\notag\\*
&=\left\| \sum_{j=1}^\infty \lambda_j\left\{\int_v^s \varsigma_{j,s}(u)(z,\cdot) dB_j^H(z)+\int_0^v \left[\varsigma_{j,s}(u)(z,\cdot)-
\varsigma_{j,v}(u)(z,\cdot)\right]    dB_j^H(z)\right\}\right\|_2
\notag\\
&\leq \left\| \sum_{j=1}^\infty \lambda_j\int_v^s \varsigma_{j,s}(u)(z,\cdot) dB_j^H(z)\right\|_2
\notag\\*
&\quad+\left\|\sum_{j=1}^\infty \lambda_j\int_0^v \left[\varsigma_{j,s}(u)(z,\cdot)-
\varsigma_{j,v}(u)(z,\cdot)\right]    dB_j^H(z)\right\|_2.
\label{eq:A-A}
\end{align}
For the first term in the right-hand side of \eqref{eq:A-A} we have that
\begin{multline*}
\left\| \sum_{j=1}^\infty \lambda_j\int_v^s \varsigma_{j,s}(u)(z,\cdot) dB_j^H(z)\right\|_2
\\*
\leq C\,\xi\sup_{j\in \mathbb{N}}\int_v^s \left(\frac{ \left\|\varsigma_{j,s}(u)(z,\cdot)\right\|_2 }{(z-v)^\sigma} + \int_v^z  \frac{    \left\| \varsigma_{j,s}(u)(z,\cdot)- \varsigma_{j,s}(u)(r,\cdot)\right\|_2}{(z-r)^{\sigma+1}}\,dr\right)dz.
\end{multline*}
Then, using Lemma $\ref{lem8}$ and the same technique as used to prove (\ref{Ih2}), we get 
\begin{multline}\label{eq1}
 \left\| \sum_{j=1}^\infty \lambda_j\int_v^s \varsigma_{j,s}(u)(z,\cdot)\, dB_j^H(z)\right\|_2^2
\\*
\leq C\,\xi^2 (s-v) \Bigg(\sup_{r \in [v,s]} \left\| u(r, \cdot ) \right\|_2^2+1
+   \int_v^s  \left( \int_v^z\frac{\left\|u(z,\cdot)-u(r,\cdot)\right\|_2  }{(z-r)^{\sigma+1}}\,dr  \right)^2  \,dz \Bigg).
\end{multline}

For the second term in the right-hand side of \eqref{eq:A-A} we have that
\begin{align*}
\MoveEqLeft[1]
\left\|\sum_{j=1}^\infty \lambda_j\int_0^v \left[\varsigma_{j,s}(u)(z,\cdot)-\varsigma_{j,v}(u)(z,\cdot)\right] dB_j^H(z)\right\|_2
\\*
&\leq C\,\xi \Bigg( \sup_{j\in \mathbb{N}} \int_0^v \frac{ \left\|\varsigma_{j,s}(u)(z,\cdot)-\varsigma_{j,v}(u)(z,\cdot)\right\|_2 }{z^\sigma} dz
\\
&\quad + \sup_{j\in \mathbb{N}} \int_0^v\!\! \int_0^z  \frac{    \Big\| [\varsigma_{j,s}(u)(z,\cdot)-\varsigma_{j,v}(u)(z,\cdot)]-[\varsigma_{j,s}(u)(r,\cdot)- \varsigma_{j,v}(u)(r,\cdot)]\Big\|_2}{(z-r)^{\sigma+1}}\,drdz \Bigg)
\\
& \eqqcolon C\,\xi \Big(  {\mathcal L}_{1} + {\mathcal L}_{2} \Big) . 
\end{align*}
Using Lemma \ref{lem9}, we obtain  
\[
\left\| \varsigma_{j,s}(u)(z,\cdot)- \varsigma_{j,v}(u)(z,\cdot) \right\|_2 \leq\,C \,(v-z)^{-\frac{\delta}{2}}\,(s-v)^{\frac{\delta}{2}}
\left (\left\|u(z,\cdot)\right\|_2+1\right).
\]
Thus, 
\begin{align*}
{\mathcal L}_1 & \leq C (s-v)^{\frac{\delta}{2}}
\int_0^v \frac{ (v-z)^{-\frac{\delta}{2}}}{z^\sigma}
\left (\left\|u(z,\cdot)\right\|_2+1\right)\,dz
\\
& \leq  C (s-v)^{\frac{\delta}{2}} 
\left(\sup_{z \in [0,v]} \left\|u(z,\cdot)\right\|_2+1\right )
\int_0^v \frac{ (v-z)^{-\frac{\delta}{2}}}{z^\sigma}
\,dz
\\
&=C (s-v)^{\frac{\delta}{2}} 
\left(\sup_{z \in [0,v]} \left\|u(z,\cdot)\right\|_2+1\right )
v^{-\delta /2 + 1 - \sigma }
B\left (1-\tfrac\delta2,1-\sigma\right )
\\
&\le C (s-v)^{\frac{\delta}{2}}\left(\sup_{z \in [0,v]} \left\|u(z,\cdot)\right\|_2+1\right )
\end{align*}
because $\sigma < 1/2$ and $\delta <  1$.

Now by Lemma \ref{lem9}, we have
\begin{align*}
\MoveEqLeft\left\| \varsigma_{j,s}(u)(z,\cdot)- \varsigma_{j,v}(u)(z,\cdot)-\varsigma_{j,s}(u)(r,\cdot)+ \varsigma_{j,v}(u)(r,\cdot) \right\|_2
\\
&\leq C \Bigl( (s-v)^{\frac{\delta}{2}} (v-z)^{-\frac{\delta}{2}}\left\| u(z,\cdot)-u(r,\cdot) \right\|_2\\
&\quad+ (s-v)^{\frac{\delta}{2}}(v-z)^{-\delta}(z-r)^{\frac{\delta}{2}} \left (\left\|u(r,\cdot) \right\|_2+1\right ) \Bigr).
\end{align*}
Thus,
\begin{align}
{\mathcal L}_2  &=  \sup_{j\in \mathbb{N}} \int_0^v\!\! \int_0^z  \frac{    \Big\| [\varsigma_{j,s}(u)(z,\cdot)-\varsigma_{j,v}(u)(z,\cdot)]-[\varsigma_{j,s}(u)(r,\cdot)- \varsigma_{j,v}(u)(r,\cdot)]\Big\|_2}{(z-r)^{\sigma+1}}\,drdz
\notag\\
& \le  C  \int_0^v\!\! \int_0^z  \frac{    (s-v)^{\frac{\delta}{2}} (v-z)^{-\frac{\delta}{2}}\,\left\| u(z,\cdot)-u(r,\cdot) \right\|_2}{(z-r)^{\sigma+1}}\,drdz
\notag\\
&\quad +  C  \int_0^v \!\!\int_0^z  \frac{ (s-v)^{\frac{\delta}{2}}(v-z)^{-\delta}(z-r)^{\frac{\delta}{2}} \left (\left\|u(r,\cdot) \right\|_2+1\right )  }{(z-r)^{\sigma+1}}\,drdz
\notag\\
&\eqqcolon  C (s-v)^{\frac{\delta }{2}}\left ( {\mathcal L}_{2,1} + {\mathcal L}_{2,2}\right ).
\label{eq:L_2}
\end{align}

By Fubini's theorem, we have 
\begin{align*}
{\mathcal L}_{2,2} &= 
\int^v_0\!\!\int_0^z    (z-r)^{\frac{\delta}{2} - \sigma-1} \,(v-z)^{-\delta}\,\left (\left\|u(r,\cdot) \right\|_2+1\right )\,dr\,dz
\\
&=\int^v_0 \left (\left\|u(r,\cdot) \right\|_2+1\right ) \left(\int_r^v(z-r)^{\frac{\delta}{2} - \sigma-1} \,(v-z)^{-\delta} \,dz\right) \,dr.
\end{align*}
Since $2\sigma<\delta<1$, we see that
\[
\int_r^v(z-r)^{\frac{\delta}{2} - \sigma-1} \,(v-z)^{-\delta} \,dz
= B\left(1-\delta,\tfrac{\delta}{2} - \sigma\right) (v-r)^{-\frac{\delta}{2}-\sigma}.
\]
 Therefore,
 \begin{align}
  {\mathcal L}_{2,2} &\le C\int_0^v \left (\left\|u(r,\cdot) \right\|_2+1\right ) (v-r)^{-\frac{\delta }{2}-\sigma} \,dr
\notag\\
   &\le   C \left (\sup_{r \in [0,v]} \left\|u(r,\cdot)\right\|_2 +1\right) \int_0^v  (v-r)^{-\frac{\delta }{2}-\sigma} \,dr
\notag\\
 & \le C v^{1-\frac{\delta }{2}-\sigma}\left (\sup_{r \in [0,v]} \left\|u(r,\cdot)\right\|_2 +1\right)
\le C \left (\sup_{r \in [0,v]} \left\|u(r,\cdot)\right\|_2 +1\right). 
\label{eq:L_22}
 \end{align}
 
It follows from \eqref{eq:L_2} and \eqref{eq:L_22} that
\[
{\mathcal L}_2 \le C (s-v)^{\frac{\delta }{2}} \left (\int_0^v \int_0^z  \frac{  (v-z)^{-\frac{\delta}{2}}\,\left\| u(z,\cdot)-u(r,\cdot) \right\|_2}{(z-r)^{\sigma+1}}\,drdz
+ \sup_{r \in [0,v]} \left\|u(r,\cdot)\right\|_2 +1\right),
\]
and consequently,
\begin{multline}\label{eq:A-2}
\left\| \sum_{j=1}^\infty \lambda_j\int_0^v \left[\varsigma_{j,s}(u)(z,\cdot)-\varsigma_{j,v}(u)(z,\cdot)\right] dB_j^H(z)\right\|_2 
  \leq  C\,\xi \Big(  {\mathcal L}_{1} + {\mathcal L}_{2} \Big)
  \\
\leq C\,\xi(s-v)^{\frac{\delta}{2}} \Bigg( \sup_{r \in [0,v]} \left\|u(r,\cdot)\right\|_2+1
 +  \int_0^v \int_0^z\frac{  (v-z)^{-\frac{\delta}{2}}\,\left\| u(z,\cdot)-u(r,\cdot) \right\|_2}{(z-r)^{\sigma+1}}\,drdz    \Bigg).
\end{multline} 

Hence, combining \eqref{eq:A-A}--\eqref{eq:A-2}, we obtain
\begin{align*}
\MoveEqLeft
\left\|(\A u)(s,\cdot)- (\A u)(v,\cdot)\right\|_2\leq\,C\,\xi\Bigg[
(s-v)^{1/2} \left (\sup_{r \in [v,s]} \left\|u(r,\cdot)\right\|_2+1\right ) 
 \\
& +  (s-v)^{1/2} \Bigg( \int_v^s \Big( \int_v^z  \frac{    \left\| u(z,\cdot)- u(r,\cdot)\right\|_2}{(z-r)^{\sigma+1}}\;dr  \Big)^2\,dz \Bigg)^{1/2}
\\
& + (s-v)^{\frac{\delta}{2}}\left ( \sup_{r \in [0,v]} \left\|u(r,\cdot)\right\|_2+1\right )
\\
& + (s-v)^{\frac{\delta}{2}} \int^v_0(v-z)^{-\frac{\delta}{2}}   \int_0^z \frac{\left\| u(z,\cdot)-u(r,\cdot) \right\|_2 }{(z-r)^{\sigma+1}}\,\;dr\,dz
\Bigg].
\end{align*}
Since $(s-v)^{\frac12} = (s-v)^{\frac{1-\delta}2} (s-v)^{\frac{\delta}{2}}
\le T^{\frac{1-\delta}2} (s-v)^{\frac{\delta}{2}}$, we see that
\begin{align*}
\MoveEqLeft
\left\|(\A u)(s,\cdot)- (\A u)(v,\cdot)\right\|_2\leq\,C\,\xi (s-v)^{\frac{\delta}{2}}\Bigg[
\sup_{r \in [0,s]} \left\|u(r,\cdot)\right\|_2+1 
 \\
& + \Bigg( \int_v^s \Big( \int_v^z  \frac{    \left\| u(z,\cdot)- u(r,\cdot)\right\|_2}{(z-r)^{\sigma+1}}\;dr  \Big)^2\,dz \Bigg)^{1/2} 
\\
& + \int^v_0(v-z)^{-\frac{\delta}{2}}   \int_0^z \frac{\left\| u(z,\cdot)-u(r,\cdot) \right\|_2 }{(z-r)^{\sigma+1}}\,\;dr\,dz\Bigg].
\end{align*}

Therefore, 
\begin{equation}\label{eq:Jh}
\norm{\A u}_{\sigma,1,t}^2=
\int_0^t\left(\int_0^s\frac{\left\|(\A u)(s,\cdot)-(\A u)(v,\cdot)\right\|_2  }{(s-v)^{\sigma+1}}\,dv\right)^2 ds
\leq\,C\,\xi^2\,\int_0^t \sum_{i=1}^3H_i^2(s)\,ds,
\end{equation}
where  
\begin{align*}
H_1(s)&=\left(\sup_{r \in [0,s]} \left\|u(r,\cdot)\right\|_2 + 1\right)
\int_0^s (s-v)^{\frac\delta2-\sigma-1}  dv,
\\
H_2(s)&=\int_0^s (s-v)^{\frac\delta2-\sigma-1} \Bigg(\int_v^s \Big( \int_v^z  \frac{    \left\| u(z,\cdot)- u(r,\cdot)\right\|_2}{(z-r)^{\sigma+1}}\,dr \Big)^2 \,dz\, \Bigg)^{1/2} dv,
 \\ 
H_3(s)&=\int_0^s (s-v)^{ \frac{\delta}{2} -\sigma-1}\int^v_0(v-z)^{-\frac{\delta}{2}}   \int_0^z \frac{\left\| u(z,\cdot)-u(r,\cdot) \right\|_2 }{(z-r)^{\sigma+1}}\,\;dr\,dz\,dv.
\end{align*}

The first term can be bounded by
\begin{equation}\label{eq:H_1}
H_1(s)=\left(\sup_{r \in [0,s]} \left\|u(r,\cdot)\right\|_2 + 1\right)
s^{\frac\delta2-\sigma} 
\le C \left (\sup_{r\in[0,s]}\left\|u(r,\cdot)\right\|_2 + 1\right ),
\end{equation}
because $\frac\delta2>\sigma$.

Concerning $H_2$, we have
\begin{equation}\label{eq:H_2}
H_2(s)\le \norm{u}_{\sigma,1,s} \int_0^s (s-v)^{\frac\delta2-\sigma-1} dv  
\le C \norm{u}_{\sigma,1,s}.
\end{equation}

For the last one, applying Fubini's theorem, we get
\begin{align*}
 H^2_3(s) &=\left[\int_0^s \left(\int^s_z (s-v)^{ \frac{\delta}{2} -\sigma-1} (v-z)^{-\frac{\delta}{2}} \,dv\right)  \int_0^z \frac{\left\| u(z,\cdot)-u(r,\cdot) \right\|_2 }{(z-r)^{\sigma+1}}\,\;dr\,\,dz\right]^2
 \\
& \leq B\left (1-\tfrac\delta2, \tfrac\delta2 - \sigma \right )^2 \left[\int_0^s  (s-z)^{-\sigma} \int_0^z \frac{\left\| u(z,\cdot)-u(r,\cdot) \right\|_2 }{(z-r)^{\sigma+1}}\,\;dr\,dz\right]^2.
\end{align*}
Since $\sigma < 1/2$, we can write, applying the Cauchy--Schwartz inequality,
\begin{equation}\label{eq:H_3}
H^2_3(s) \leq  C \int_0^s  (s-z)^{-2\sigma} \,dz\int_0^s \left(\int_0^z \frac{\left\| u(z,\cdot)-u(r,\cdot) \right\|_2 }{(z-r)^{\sigma+1}}\,\;dr\,\right)^2\,dz
\le C\norm{u}_{\sigma,1,s}^2. 
\end{equation}

Finally, combining \eqref{eq:Jh}--\eqref{eq:H_3}, we get
\begin{equation}
\label{Jh}
\norm{\A u}_{\sigma,1,t}^2
\leq \,C\,\xi^2 \left( \int_0^t \left\| u\right\|^2_{\sigma,2,s}\,ds+1\right )
\leq \,C\,\xi^2 \left(\left\| u\right\|^2_{\sigma,2,t}+1\right),
\end{equation}
because $\left\| u\right\|^2_{\sigma,2,s}\leq \left\| u\right\|^2_{\sigma,2,t}$ for $s\leq t$.

Hence, from \eqref{lem7}, \eqref{eq:Au-1} and \eqref{Jh} we get the result.
\end{proof}
\begin{proposition}\label{prop12}
Let
$u,\tilde{u}\in  \mathcal{B}^{\sigma,2}\left([0,T];L^2(\R) \right)$.
Then for all $t\in[0,T]$,
\[
\left\|\A u - \A\tilde{u} \right\|^2_{\sigma,2,t}
\leq\,C\xi^2\,\int_0^t \left\|u-\tilde{u}\right\|^2_{\sigma,2,s}\,ds.
\]
\end{proposition}
\begin{proof}
Recall that $\sigma \in \left(1-H,\frac{1}{2}\right)$.
Fix $\delta\in\bigl(\max\{\sigma,\frac13\},1\bigr)$.

By the definition of $\A$, we can write
\begin{align*}
\MoveEqLeft
\left\|(\A u)(t,\cdot)-(\A\tilde{u})(t,\cdot) \right\|_2
= \left\|\sum_{j=1}^\infty\lambda_j\int_0^t \left[ \varsigma_{j,t}(u)(s,\cdot)-\varsigma_{j,t}(\tilde{u})(s,\cdot)\right]dB^{H}_j(s)\right\|_2
\\
&\leq C\, \xi \sup_{j\in\mathbb{N}}
\int_0^t \Bigg(\frac{ \left\|\varsigma_{j,t}(u)(s,\cdot)-\varsigma_{j,t}(\tilde{u})(s,\cdot)\right\|_2 }{s^\sigma}
\\
&\quad + \int_0^s  \frac{    \left\| \varsigma_{j,t}(u)(s,\cdot)-\varsigma_{j,t}(\tilde{u})(s,\cdot)- \varsigma_{j,t}(u)(r,\cdot)+\varsigma_{j,t}(\tilde{u})(r,\cdot)\right\|_2}{(s-r)^{\sigma+1}}\,dr\Bigg)ds.
\end{align*}
It follows from Lemma \ref{lem9b} that 
\begin{align*}
\MoveEqLeft
\left\|(\A u)(t,\cdot)- (\A\tilde{u})(t,\cdot) \right\|_2
\leq C\, \xi \Bigg\{\int_0^t 
\Bigg(\frac{ \left\|u(s,\cdot)-\tilde{u}(s,\cdot)\right\|_2 }{s^\sigma}
\\
&+ \int_0^s 
\frac{(t-s)^{-\delta}\, \left\| u(r,\cdot)-\tilde{u}(r,\cdot) \right\|_2  }{(s-r)^{\sigma+1-\delta}}\,dr\Bigg)ds
\\
&+ \int_0^t \int_0^s \frac{ \left\|u(s,\cdot)-\tilde{u}(s,\cdot)- u(r,\cdot)+\tilde{u}(r,\cdot) \right\|_2      }{(s-r)^{\sigma+1}}\,dr\,ds \Bigg\},
\end{align*}
where the right-hand side is finite, because
$u,\tilde{u}\in  \mathcal{B}^{\sigma,2}\left([0,T];L^2(\R) \right)$.

First, the Cauchy--Schwarz inequality implies that for every
$\sigma < \frac{1}{2}$
\begin{align*}
\int_0^t \frac{ \left\|u(s,\cdot)-\tilde{u}(s,\cdot)\right\|_2 }{s^\sigma}ds &\leq \left(\int_0^t s^{-2\sigma}\,ds\right)^{1/2} \left(\int_0^t  \left\|u(s,\cdot)-\tilde{u}(s,\cdot)\right\|_2^2\,ds \right)^{1/2}
\\
&\leq C \left(\int_0^t  \sup_{r\in(0,s)}\left\|u(r,\cdot)-\tilde{u}(r,\cdot)\right\|_2^2\,ds \right)^{1/2}
\end{align*}
Second, it follows from Fubini's theorem and the Cauchy--Schwarz inequality that for every $\sigma < \delta <1$
\begin{align*}
\MoveEqLeft\int_0^t\!\!\int_0^s 
\frac{(t-s)^{-\delta}\, \left\| u(r,\cdot)-\tilde{u}(r,\cdot) \right\|_2  }{(s-r)^{\sigma+1-\delta}}\,dr\,ds
\\
&= \int_0^t  \left\| u(r,\cdot)-\tilde{u}(r,\cdot) \right\|_2 \int_r^t 
(t-s)^{-\delta}(s-r)^{-\sigma-1+\delta}\,ds\,dr
\\
&= B(1-\delta,\delta-\sigma)\int_0^t (t-r)^{-\sigma}\left\| u(r,\cdot)-\tilde{u}(r,\cdot) \right\|_2 
\,dr
\\
&\leq C \left(\int_0^t (t-r)^{-2\sigma}\,dr\right)^{\frac12} \left(\int_0^t \,\sup_{r\in(0,s)} \left\|u(r,\cdot)-\tilde{u}(r,\cdot)\right\|_2^2\,ds \right)^{\frac12}
\\
&\leq C \left(\int_0^t \,\sup_{r\in(0,s)} \left\|u(r,\cdot)-\tilde{u}(r,\cdot)\right\|_2^2\,ds \right)^{\frac12}.
\end{align*}
Hence,
\begin{multline*}
\left\|(\A u)(t,\cdot)-(\A\tilde{u})(t,\cdot) \right\|_2
\leq C\, \xi \Bigg\{ \left(\int_0^t \,\sup_{r\in(0,s)} \left\|u(r,\cdot)-\tilde{u}(r,\cdot)\right\|_2^2\,ds \right)^{1/2}
\\
+t\int_0^t \left(\int_0^s \frac{ \left\|u(s,\cdot)-\tilde{u}(s,\cdot)- u(r,\cdot)+\tilde{u}(r,\cdot) \right\|_2      }{(s-r)^{\sigma+1}}\,dr\right)^2\,ds \Bigg\}.
\end{multline*}
Therefore, we obtain
\[
\sup_{s\in(0,t)}\left\|(\A u)(s,\cdot)-(\A\tilde{u})(s,\cdot) \right\|^2_{2}
\leq C\, \xi^2\int_0^t \left\|u-\tilde{u}\right\|^2_{\sigma,2,s}\,ds , 
\]

On the other hand, we have that for $r<s$
\begin{align*}
\MoveEqLeft[1]
\left\|(\A u)(s,\cdot)-(\A\tilde{u})(s,\cdot)-(\A u)(r,\cdot)+(\A\tilde{u})(r,\cdot) \right\|_2
\\*
&\leq\, \left\|\sum_{j=1}^\infty\lambda_j
\int_r^s \left[ \varsigma_{j,s}(u)(v,\cdot)-\varsigma_{j,s}(\tilde{u})(v,\cdot)\right]dB^H_j(v)\right\|_2
\\
&\quad+ \left\|\sum_{j=1}^\infty\lambda_j
\int_0^r \left[ \varsigma_{j,s}(u)(v,\cdot)-\varsigma_{j,s}(\tilde{u})(v,\cdot)-
\varsigma_{j,r}(u)(v,\cdot)+\varsigma_{j,r}(\tilde{u})(v,\cdot)\right]dB^H_j(v)\right\|_2
\\
&\eqqcolon \mathcal{J}_1+\mathcal{J}_2.
\end{align*}
The first term can be bounded as follows
\begin{align*}
\mathcal{J}_1&=\left\|\sum_{j=1}^\infty\lambda_j
\int_r^s \left[ \varsigma_{j,s}(u)(v,\cdot)-\varsigma_{j,s}(\tilde{u})(v,\cdot)\right]dB^H_j(v)\right\|_2
\\
&\leq  C\, \xi \sup_{j\in\mathbb{N}}
\int_r^s \Bigg(\frac{ \left\|\varsigma_{j,s}(u)(v,\cdot)-\varsigma_{j,s}(\tilde{u})(v,\cdot)\right\|_2 }{(v-r)^\sigma}
\\
&\quad+ \int_r^v  \frac{    \left\| \varsigma_{j,s}(u)(v,\cdot)-\varsigma_{j,s}(\tilde{u})(v,\cdot)- \varsigma_{j,s}(u)(z,\cdot)+\varsigma_{j,s}(\tilde{u})(z,\cdot)\right\|_2}{(v-z)^{\sigma+1}}\,dz\Bigg)dv.
\end{align*}
By Lemma \ref{lem9b}, 
\begin{multline*}
\mathcal{J}_1\leq C\xi \int_r^s \Biggl( \frac{ \left\|u(v,\cdot)-\tilde{u}(v,\cdot)\right\|_2 }{(v-r)^\sigma}+  \int_r^v  (s-v)^{-\delta}\frac{    \left\| u(z,\cdot)-\tilde{u}(z,\cdot)\right\|_2}{(v-z)^{\sigma+1-\delta}}\,dz
\\
+ \int_r^v  \frac{    \left\| u(v,\cdot)-\tilde{u}(v,\cdot)- u(z,\cdot)+\tilde{u}(z,\cdot)\right\|_2}{(v-z)^{\sigma+1}}\,dz\Biggr)dv.
\end{multline*}
The term $\mathcal{J}_2$ can be written and bounded as follows 
\begin{align*}
\mathcal{J}_2 &= \left\|\sum_{j=1}^\infty\lambda_j
\int_0^r \left[ \varsigma_{j,s}(u)(v,\cdot)-\varsigma_{j,s}(\tilde{u})(v,\cdot)-
\varsigma_{j,r}(u)(v,\cdot)+\varsigma_{j,r}(\tilde{u})(v,\cdot)\right]dB^H_j(v)\right\|_2
\\
&= \left\|\sum_{j=1}^\infty\lambda_j
\int_0^r \left[ \varsigma^*_{j,s,r}(u)(v,\cdot)-\varsigma^*_{j,s,r}(\tilde{u})(v,\cdot)\right]dB^H_j(v)\right\|_2
\\
&\leq  C\, \xi \sup_{j\in\mathbb{N}}
\int_0^r \Bigg(\frac{ \left\|\varsigma^*_{j,s,r}(u)(v,\cdot)-\varsigma^*_{j,s,r}(\tilde{u})(v,\cdot)\right\|_2 }{v^\sigma}
\\
&\quad+ \int_0^v  \frac{    \left\| \varsigma^*_{j,s,r}(u)(v,\cdot)-\varsigma^*_{j,s,r}(\tilde{u})(v,\cdot)- \varsigma^*_{j,s,r}(u)(z,\cdot)+\varsigma^*_{j,s,r}(\tilde{u})(z,\cdot)\right\|_2}{(v-z)^{\sigma+1}}\,dz\Bigg)dv.
\end{align*}
By Lemma \ref{lem10}, we get
\begin{align*}
\mathcal{J}_2 &\leq  C\, \xi \int_0^r \biggl(  (s-r)^{\frac{\delta}{2}} (r-v)^{- \frac{\delta}{2}}   \frac{ \left\|u(v,\cdot)-\tilde{u}(v,\cdot)\right\|_2 }{v^\sigma}
\\
&\quad+ \int_0^v (s-r)^{\frac{\delta}{2}}\,(r-v)^{-\frac{\delta}{2}}\, \frac{    \left\| u(z,\cdot)-\tilde{u}(z,\cdot)\right\|_2}{(v-z)^{\sigma+1- \frac{\delta}{2}   }}\,dz
\\
&\quad + \int_0^v (s-r)^{\frac{\delta}{2}}\,(r-v)^{-\frac{\delta}{2}}\, \frac{    \left\|u(v,\cdot)-\tilde{u}(v,\cdot)- u(z,\cdot)+\tilde{u}(z,\cdot)\right\|_2}{(v-z)^{\sigma+1}}\,dz\biggr)dv.
\end{align*}
Therefore 
\begin{multline*}
\int_0^t \left(\int_0^s \frac{ \left\|(\A u)(s,\cdot)-(\A\tilde{u})(s,\cdot)-(\A u)(r,\cdot)+(\A\tilde{u})(r,\cdot) \right\|_2      }{(s-r)^{\sigma+1}}\,dr\right)^2\,ds
\\
= C\, \xi^2\int_0^t \sum_{i=1}^6 G_i^2(s)\,ds,
\end{multline*}
where 
\begin{align*}
G_1(s)&\coloneqq \int_0^s (s-r)^{-\sigma-1}\int_r^s  \frac{ \left\|u(v,\cdot)-\tilde{u}(v,\cdot)\right\|_2 }{(v-r)^\sigma}\,dv\,dr,
\\ 
G_2(s)&\coloneqq \int_0^s (s-r)^{-\sigma-1}\int_r^s\!\! \int_r^v  (s-v)^{-\delta}\frac{    \left\| u(z,\cdot)-\tilde{u}(z,\cdot)\right\|_2}{(v-z)^{\sigma+1-\delta}}\,dz\,dv\,dr,
\\ 
G_3(s)&\coloneqq \int_0^s (s-r)^{-\sigma-1}\int_r^s\!\!    \int_r^v  \frac{    \left\| u(v,\cdot)-\tilde{u}(v,\cdot)- u(z,\cdot)+\tilde{u}(z,\cdot)\right\|_2}{(v-z)^{\sigma+1}}\,dz\,dv\,dr,
\\ 
G_4(s)&\coloneqq \int_0^s (s-r)^{\frac{\delta}{2}-\sigma-1}\int_0^r  (r-v)^{- \frac{\delta}{2}}   \frac{ \left\|u(v,\cdot)-\tilde{u}(v,\cdot)\right\|_2 }{v^\sigma}\,dv\,dr,
\\ 
G_5(s)&\coloneqq \int_0^s (s-r)^{\frac{\delta}{2}-\sigma-1}\int_0^r\!\! \int_0^v (r-v)^{-\frac{\delta}{2}}\, \frac{    \left\| u(z,\cdot)-\tilde{u}(z,\cdot)\right\|_2}{(v-z)^{\sigma+1- \frac{\delta}{2}   }}\,dz\,dv\,dr,
\\  
G_6(s)&\coloneqq \int_0^s (s-r)^{\frac{\delta}{2}-\sigma-1}\\
&\quad\times\int_0^r\!\! \int_0^v (r-v)^{-\frac{\delta}{2}}\, \frac{    \left\|u(v,\cdot)-\tilde{u}(v,\cdot)- u(z,\cdot)+\tilde{u}(z,\cdot)\right\|_2}{(v-z)^{\sigma+1}}\,dz\,dv\,dr.
\end{align*}
Let us bound each of terms $G_i$, $i=1,\dots,6$.
For every $\sigma \in (0,1/2)$ 
\begin{align*}
G_1(s) &\leq  \sup_{z\in[0,s]} \left\|u(z,\cdot)-\tilde{u}(z,\cdot)\right\|_2 \left(\int_0^s (s-r)^{-\sigma-1}\int_r^s  (v-r)^{-\sigma}\,dv\,dr\right)
\\
&=C\, \sup_{z\in[0,s]} \left\|u(z,\cdot)-\tilde{u}(z,\cdot)\right\|_2.
\end{align*}
and by the same technique we get for every $\delta > \sigma$
\[
G_2(s) \leq \,C\, \sup_{z\in[0,s]} \left\|u(z,\cdot)-\tilde{u}(z,\cdot)\right\|_2.
\]
By the Cauchy--Schwartz inequality, we obtain
\[
G_3(s)\leq \int_0^s\! (s-r)^{-\sigma-\frac{1}{2}}\left(\int_r^s \!\left(   \int_r^v\!  \frac{    \norm{u(v,\cdot)-\tilde{u}(v,\cdot)- u(z,\cdot)+\tilde{u}(z,\cdot)}_2}{(v-z)^{\sigma+1}}\,dz\right)^2 \!dv\right)^{\frac12}\!dr.
\]
Similarly to \eqref{eq:H_2}, we get the bound
\[
G_3(s)\leq \, C \left[\int_0^s \left(   \int_0^v  \frac{    \left\| u(v,\cdot)-\tilde{u}(v,\cdot)- u(z,\cdot)+\tilde{u}(z,\cdot)\right\|_2}{(v-z)^{\sigma+1}}\,dz\right)^2 \,dv
\right]^{1/2}.
\]
It is not hard to see that  
\begin{align*}
G_4(s)&\leq \sup_{z\in[0,s]} \left\|u(z,\cdot)-\tilde{u}(z,\cdot)\right\|_2 \int_0^s (s-r)^{\frac{\delta}{2}-\sigma-1}\int_0^r  (r-v)^{- \frac{\delta}{2}}v^{-\sigma}\,dv\,dr
\\
&\leq C\, \sup_{z\in[0,s]} \left\|u(z,\cdot)-\tilde{u}(z,\cdot)\right\|_2. 
\end{align*}
In the same way,
\[
G_5(s)\leq C\, \sup_{z\in[0,s]} \left\|u(z,\cdot)-\tilde{u}(z,\cdot)\right\|_2. 
\]
Finally, the term $G_6$ can be bounded similarly to \eqref{eq:H_3}. We obtain
\[
G_6(s)\leq C \left[\int_0^s \left(   \int_0^v  \frac{    \left\| u(v,\cdot)-\tilde{u}(v,\cdot)- u(z,\cdot)+\tilde{u}(z,\cdot)\right\|_2}{(v-z)^{\sigma+1}}\,dz\right)^2 \,dv
\right]^{1/2}.
\]
Then we have 
\[ 
\left\|(\A u)(s,\cdot)-(\A\tilde{u})(s,\cdot) \right\|^2_{\sigma,2,t}\leq C \xi^2\int_0^t \left\|u-\tilde{u}\right\|^2_{\sigma,2,s}\,ds. 
\qedhere\]
\end{proof}

\subsection{Existence and uniqueness of mild solution}
\begin{theorem}
For every $H\in(\frac{1}{2},1),$ there exists a unique mild solution to the problem \eqref{e:2}.
\end{theorem}
\begin{proof}
\underline{Existence.}
We define the following sequence of random processes $u_p\colon [0,T]\to L^2(\R)$:
\[
u_0\equiv0,
\quad
u_{p+1}=\A u_p,\: p\ge1.
\]

Reasoning by induction and using Proposition \ref{prop11}, we easily get that
$u_p \in \mathcal{B}^{\sigma,2}\left([0,T];L^2(\R) \right)$ a.\,s. for every $p \ge 0$.
By Proposition~\ref{prop12}, for any $p\ge1$,
\[
\norm{u_{p+1}-u_p}_{\sigma,2,T}^2
=\left\|\A u_p - \A u_{p-1} \right\|^2_{\sigma,2,T}
\leq\,C\xi^2\,\int_0^T \left\|u_p-u_{p-1}\right\|^2_{\sigma,2,s}\,ds.
\]
By induction, we get 
\[
\norm{u_{p+1}-u_p}_{\sigma,2,t}^2
\leq\,C^p\xi^{2p}\,\int_0^T\int_0^{s_1}\dots\int_0^{s_{p-1}}
\left\|u_1-u_0\right\|^2_{\sigma,2,s_p}\,ds_p\dots ds_2\,ds_1.
\]
Since $u_0\equiv0$, we see that by Proposition~\ref{prop11},
\[
\left\|u_1-u_0\right\|^2_{\sigma,2,s_p}
=\left\|u_1\right\|^2_{\sigma,2,s_p}
=\left\|\A u_0\right\|^2_{\sigma,2,s_p}
\le C \xi^2 
\left(\norm{u_0}_{\sigma,2,s_p}^2+1\right)
=C \xi^2.
\]
Hence,
\[
\norm{u_{p+1}-u_p}_{\sigma,2,T}^2
\leq\left(C\xi^2\right)^{p+1}\,\int_0^T\int_0^{s_1}\dots\int_0^{s_{p-1}}
ds_p\dots ds_2\,ds_1
= \left(C\xi^2\right)^{p+1} \frac{T^p}{p!}.
\]
Then for $m>n\ge0$ we get
\begin{align*}
\norm{u_m-u_n}_{\sigma,2,t}
&=\norm{\sum_{p=n}^{m-1}(u_{p+1}-u_p)}_{\sigma,2,t}
\le\sum_{p=n}^{m-1}\norm{u_{p+1}-u_p}_{\sigma,2,t}
\\
&\le\sum_{p=n}^{m-1}\left(\frac{\left(C\xi^2\right)^{p+1}T^p}{p!}\right)^{1/2}
\to 0
\end{align*}
a.\,s. as $m,n\to\infty$.

Therefore, $\{u_{p},p\geq 0\}$ is a Cauchy sequence in $\mathcal{B}^{\sigma,2}\left([0,T];L^2(\R) \right)$ a.\,s.
Then there exists a process $u_\infty \in \mathcal{B}^{\sigma,2}\left([0,T];L^2(\R) \right)$
such that
\begin{equation}\label{eq:lim}
\norm{u_p-u_\infty}_{\sigma,2,T}\to0\quad \text{a.\,s., as }p\to\infty.
\end{equation}

Now we prove that $u_\infty$ is a mild solution.
For any $p\ge0$,
\begin{align*}
\norm{u_\infty-\A u_\infty}_{\sigma,2,T}^2
&\le 2\norm{u_\infty-u_{p+1}}_{\sigma,2,T}^2 + 2\norm{u_{p+1}-\A u_\infty}_{\sigma,2,T}^2
\\
&= 2\norm{u_\infty-u_{p+1}}_{\sigma,2,T}^2 + 2\norm{\A u_p-\A u_\infty}_{\sigma,2,T}^2
\\
&\le 2\norm{u_\infty-u_{p+1}}_{\sigma,2,T}^2 + C\xi^2\,\int_0^T\norm{u_p-u_\infty}_{\sigma,2,s} \,ds
\\
&\le 2\norm{u_\infty-u_{p+1}}_{\sigma,2,T}^2
+ C\xi^2 T \norm{u_p-u_\infty}_{\sigma,2,T},
\end{align*}
a.\,s., by Proposition~\ref{prop12}.
Letting $p\to\infty$ and taking into account \eqref{eq:lim}, we get
that $u_\infty=\A u_\infty$ a.\,s.
Hence, $u_\infty$ is a mild solution.

\underline{Uniqueness.}
Let $u$ and $\tilde u$ be two mild solutions.
Then $u=\A u$ and $\tilde u =\A\tilde u$ a.\,s., and for all $t\in[0,T]$,
\[
\left\| u - \tilde{u} \right\|^2_{\sigma,2,t}
=\left\|\A u - \A\tilde{u} \right\|^2_{\sigma,2,t}
\leq\,C\xi^2\,\int_0^t \left\|u-\tilde{u}\right\|^2_{\sigma,2,s}\,ds
\qquad\text{a.\,s.},
\]
by Proposition \ref{prop12}.
Then, by Gronwall's lemma,
$\left\| u - \tilde{u} \right\|^2_{\sigma,2,T}=0$ a.\,s., which means the uniqueness of the mild solution.
\end{proof}

\appendix

\section{}
In this appendix we establish upper bounds for $L^2$-norms of the functions $\varsigma_{j,t}(u)$ and $\varsigma^*_{j,t,s}(u)$ and their differences (see \eqref{eq:zeta} and \eqref{eq:zeta*} for the definitions of that functions).
The results of the appendix are used in the proofs of Propositions~\ref{prop11} and \ref{prop12}.

\begin{lemma}\label{lem8} 
Let  $\{u(t,\cdot),t\in[0,T]\}$ be an $L^2(\R)$-valued random process.
Then 
\begin{enumerate}[(i)]
\item
for all $0<s<t<T$,
\[
\sup_{j\in \mathbb{N}}\left\| \varsigma_{j,t}(u)(s,\cdot)\right\|_2 \leq \, C\, \left (\left\|u(s,\cdot)\right\|_2+1\right );
\]

\item
for all $0<r<s<t<T$, and for any $\delta\in(\frac13,1)$,
\begin{multline*}
\sup_{j\in \mathbb{N}}\left\| \varsigma_{j,t}(u)(s,\cdot)- \varsigma_{j,t}(u)(r,\cdot) \right\|_2
\\*
\leq \, C\left\{\left\|u(s,\cdot)-u(r,\cdot) \right\|_2+(t-s)^{-\frac{\delta}{2}}\,(s-r)^{\frac{\delta}{2}}\left(\left\|u(r,\cdot)\right\|_2+1\right)\right\}.
\end{multline*}
\end{enumerate}
\end{lemma}
\begin{proof}
Let us start with the assertion $(i)$ and produce the following transformations:
\[
\left\| \varsigma_{j,t}(u)(s,\cdot)\right\|_2^2
= \int_\R \left|\varsigma_{j,t}(u)(s,x)\right|^2\,dx
=\int_\R \left|\int_\R G(t-s, x, y)\,h(u(s,y))\;e_j(y)\,dy\right|^2\,dx.
\]
It follows from the H\"older inequality that
\[
\left\| \varsigma_{j,t}(s,\cdot)\right\|_2^2 \leq \int_\R\left( \int_\R \left|G(t-s, x, y)\right|dy \int_\R \left|G(t-s, x, y)\right| \left|h(u(s,y))\right| ^2|e_j(y)|^2dy\right)dx.
\]
Corollary \ref{cor2} implies that
\[
\left\| \varsigma_{j,t}(u)(s,\cdot)\right\|_2^2 \leq C \int_\R\left(  \int_\R \left|G(t-s, x, y)\right|\, \left|h(u(s,y))\right| ^2|e_j(y)|^2\,dy\right)\,dx.
\]
By \eqref{lin}, for all $z\in\R$,
\[
\abs{h(z)}^2\le  C\left(z^2 + 1\right).
\]
Therefore, applying Fubini's Theorem, we get
\begin{align*}
\left\| \varsigma_{j,t}(u)(s,\cdot)\right\|_2^2
&\leq C\int_{\R^2} \left|G(t-s, x, y)\right|\, \left(\left|u(s,y)\right|^2+1\right)|e_j(y)|^2\,dx\,dy
\\
&=  C \int_{\R} \left(\left|u(s,y)\right|^2+1\right)|e_j(y)|^2 \int_{\R} \left|G(t-s, x, y)\right|\,dx \,dy.
\end{align*}
Using again Corollary  \ref{cor2},  we get
\[
\left\| \varsigma_{j,t}(u)(s,\cdot)\right\|_2^2
\leq  C  \int_{\R} \left(\left|u(s,y)\right|^2+1\right)|e_j(y)|^2\,dy.
\]
However,
\begin{align}
\int_{\R} \left(\left|u(s,y)\right|^2+1\right)|e_j(y)|^2\,dy
&\le C  \left(\sup_{j\in\mathbb{N}}\left\|e_j\right\|_\infty^2 \norm{u(s,\cdot)}_2^2+\norm{e_j}_2^2 \right)
\notag\\
&\le C\left(\norm{u(s,\cdot)}_2^2+1 \right),
\label{eq:bound-lg}
\end{align}
since $\sup_{j\in\mathbb{N}}\left\|e_j\right\|_\infty^2$ is bounded by (\ref{ej}), and $\norm{e_j}_2=1$ due to orthonormality.
Consequently, the assertion $(i)$ follows.

Now, let us prove assertion $(ii)$. On the one hand, obviously,
 \[\left\| \varsigma_{j,t}(u)(s,\cdot)- \varsigma_{j,t}(u)(r,\cdot) \right\|_2:=\left(\int_\R \left|\varsigma_{j,t}(u)(s,x)-\varsigma_{j,t}(u)(r,x)\right|^2\,dx\right)^{1/2}.
\]
On the other hand,
\begin{align}
\MoveEqLeft
\left|\varsigma_{j,t}(u)(s,x)-\varsigma_{j,t}(u)(r,x)\right|
\notag\\*
&= \left|\int_\R\big[ G(t-s, x, y)\,h(u(s,y))\,e_j(y)- G(t-r, x, y)\,h(u(r,y))\,e_j(y)\big]\,dy\right|
\notag\\
&= \left|\int_\R G(t-s, x, y)\,\Big[h(u(s,y))-h(u(r,y))\Big]\,e_j(y)\,dy\right.
\notag\\*
&\quad+\left.\int_\R \Big[G(t-s, x, y)-G(t-r, x, y)\Big]\,h(u(r,y))\,e_j(y)\,dy\right|
\notag\\
&\leq \sup_{j\in\mathbb{N}}\left\|e_j\right\|_\infty
 \int_\R \left|G(t-s, x, y)\right|\,\left|h(u(s,y))-h(u(r,y))\right|\,dy
\notag\\*
&\quad+\int_\R \left|G(t-s, x, y)-G(t-r, x, y)\right|\,\left|h(u(r,y))\right| \abs{e_j(y)}\,dy.
\label{eq:l7-1}
\end{align}
By Holder's inequality, we obtain from \eqref{eq:l7-1} that
\begin{align}
\MoveEqLeft
\left|\varsigma_{j,t}(u)(s,x)-\varsigma_{j,t}(u)(r,x)\right|^2
\leq C \left[\int_\R \left|G(t-s, x, y)\right|\left|h(u(s,y))-h(u(r,y))\right|dy\right]^2
\notag\\*
&\quad+ C\left[\int_\R \left|G(t-s, x, y)-G(t-r, x, y)\right|\left|h(u(r,y))\right|\abs{e_j(y)}\,dy \right]^2
\notag\\
&\leq\,C \int_\R \left|G(t-s, x, y)\right|\,dy\int_\R\left|G(t-s, x, y)\right|\left|h(u(s,y))-h(u(r,y))\right|^2\,dy
\notag\\*
&\quad+\,C\int_\R \left|G(t-s, x, y)-G(t-r, x, y)\right|\,dy\notag\\*
&\quad\quad\times\int_\R  \left|G(t-s, x, y)-G(t-r, x, y)\right|\left|h(u(r,y))\right|^2\abs{e_j(y)}^2\,dy.
\label{eq:l7-2}
\end{align}
Applying  Corollaries  \ref{cor2}  and  \ref{cor4}, we get from \eqref{eq:l7-2} the following bounds
\begin{align*}
\MoveEqLeft[1]
\left|\varsigma_{j,t}(u)(s,x)-\varsigma_{j,t}(u)(r,x)\right|^2
\leq\,C \int_\R\left|G(t-s, x, y)\right|\left|h(u(s,y))-h(u(r,y))\right|^2\,dy
\\*
&\quad+ C(t-s)^{-\delta}\,(s-r)^\delta\int_\R  \left|G(t-s, x, y)-G(t-r, x, y)\right|\left|h(u(r,y))\right|^2\abs{e_j(y)}^2\,dy
\\
&\leq C\int_\R \left|G(t-s, x, y)\right|\,\left|u(s,y)-u(r,y)\right|^2\,dy
+C(t-s)^{-\delta}\,(s-r)^\delta
\\*
&\quad\times\int_\R  \left|G(t-s, x, y)-G(t-r, x, y)\right|\left(\left|u(r,y)\right|^2+1\right)\abs{e_j(y)}^2\,dy.
\end{align*}
Applying Fubini's  theorem, we get
\begin{align*}
\MoveEqLeft[1]
\int_\R\left|\varsigma_{j,t}(u)(s,x)-\varsigma_{j,t}(u)(r,x)\right|^2\,dx
\\
&\leq C\,\,\int_\R\left|u(s,y)-u(r,y)\right|^2\int_\R \left|G(t-s, x, y)\right|\,dx\,dy
\\*
&\quad+C(t-s)^{-\delta}\,(s-r)^\delta  \int_\R\left(\left|u(r,y)\right|^2 +1\right) \abs{e_j(y)}^2
\\
&\quad\quad\times\left(\int_\R  \left|G(t-s, x, y)-G(t-r, x, y)\right|\,dx\right)\,dy.
\\
&\leq C\,\left\|u(s,\cdot)-u(r,\cdot)\right\|^2_2
+ C(t-s)^{-\delta}\,(s-r)^\delta
\\
&\quad\times
\int_\R\left(\left|u(r,y)\right|^2 +1\right) \abs{e_j(y)}^2\left(\int_\R  \left|G(t-s, x, y)\right|dx+\int_\R  \left|G(t-r, x, y)\right|dx  \right)dy.
\end{align*}
Using again Corollary  \ref{cor2} and the bound \eqref{eq:bound-lg},  we get
\begin{align*}
\MoveEqLeft
\int_\R\left|\varsigma_{j,t}(u)(s,x)-\varsigma_{j,t}(u)(r,x)\right|^2\,dx
\leq C\,\left\|u(s,\cdot)-u(r,\cdot)\right\|^2_2
\\*
&\quad+C(t-s)^{-\delta}\,(s-r)^\delta  \int_\R\left(\left|u(r,y)\right|^2 +1\right) \abs{e_j(y)}^2\,dy
\\
&\leq C\,\Big[\left\|u(s,\cdot)-u(r,\cdot)\right\|^2_2+\,(t-s)^{-\delta}\,(s-r)^\delta\,\left(\left\| u(r,\cdot)\right\|^2_2+1\right) \Big].
\end{align*}
Consequently,
\begin{multline*}
\left\|\varsigma_{j,t}(u)(s,\cdot)-\varsigma_{j,t}(u)(r,\cdot)\right\|_2
\\*
\leq C\,\Big[\left\|u(s,\cdot)-u(r,\cdot)\right\|_2+(t-s)^{-\frac{\delta}{2}}\,(s-r)^{\frac{\delta}{2}}\,\left(\left\| u(r,\cdot)\right\|_2+1\right)\Big].
\end{multline*}
Hence, the statement $(ii)$ is obtained.
\end{proof}
\begin{lemma}\label{lem9}
Let  $\{u(t,\cdot),t\in[0,T]\}$ be an $L^2(\R)$-valued random process.
Then  
\begin{enumerate}[(i)]
\item
for all $0<v<s<t<T$ and $\delta  \in(\frac13,1)$,
\[
\sup_{j\in \mathbb{N}}\left\| \varsigma_{j,t}(u)(v,\cdot)- \varsigma_{j,s}(u)(v,\cdot) \right\|_2 \leq \,C(s-v)^{-\frac{\delta}{2}}\,(t-s)^{\frac{\delta}{2}}\left(\left\|u(v,\cdot)\right\|_2+1\right),
\]
\item
for all $0<r<v<s<t<T$ and $\delta \in(\frac13,1), \delta' \in(\frac15,1)$, 
\begin{align*}
\MoveEqLeft
\sup_{j\in \mathbb{N}}\left\| \varsigma_{j,t}(u)(v,\cdot)- \varsigma_{j,s}(u)(v,\cdot)-\varsigma_{j,t}(u)(r,\cdot)+ \varsigma_{j,s}(u)(r,\cdot) \right\|_2
\\
&\leq C \Bigl( (t-s)^{\frac{\delta}{2}} (s-v)^{-\frac{\delta}{2}}\,\left\| u(v,\cdot)-u(r,\cdot)\right\|_2
\\
&\quad+  (t-s)^{\frac{\delta'}{2}} (s-v)^{-\delta'} (v-r)^{\frac{\delta'}{2}} \left(\left\|u(r,\cdot)\right\|_2+1\right)   \Bigr) . 
\end{align*}
\end{enumerate}
\end{lemma}
\begin{proof}
$(i)$
First, we can produce the relations
\begin{align*}
\MoveEqLeft
\left|\varsigma_{j,t}(u)(v,x)-\varsigma_{j,s}(u)(v,x)\right|
\\*
&= \left|\int_\R\big[ G(t-v, x, y)\,h(u(v,y))\,e_j(y)- G(s-v, x, y)\,h(u(v,y))\,e_j(y)\big]\,dy\right|
\\
&= \left|\int_\R \Big[G(t-v, x, y)-G(s-v, x, y)\Big]\,h(u(v,y))\,e_j(y)\,dy\right|.
\end{align*}
By using Holder's inequality, we obtain
\begin{multline*}
\left|\varsigma_{j,t}(u)(v,x)-\varsigma_{j,s}(u)(v,x)\right|^2
\leq\,C\int_\R \left|G(t-v, x, y)-G(s-v, x, y)\right|\,dy
\\
\times\int_\R  \left|G(t-v, x, y)-G(s-v, x, y)\right|\left|h(u(v,y))\right|^2 \abs{e_j(y)}^2\,dy
\end{multline*}
Then, we can deduce from  Corollary  \ref{cor2}
\begin{align*}
\MoveEqLeft
\left|\varsigma_{j,t}(u)(v,x)-\varsigma_{j,s}(u)(v,x)\right|^2
\\*
&\leq\,C\int_\R  \left|G(t-v, x, y)-G(s-v, x, y)\right|\left|h(u(v,y))\right|^2 \abs{e_j(y)}^2\,dy
\\
&\leq\,C\int_\R  \left|G(t-v, x, y)-G(s-v, x, y)\right|\,\left(\left|u(v,y)\right|^2+1\right) \abs{e_j(y)}^2\,dy.
\end{align*}
Hence, it follows from  Fubini theorem that
\begin{multline*}
\int_\R\left|\varsigma_{j,t}(u)(v,x)-\varsigma_{j,s}(u)(v,x)\right|^2\,dx
\\*
\leq\,C  \int_\R\left(\left|u(v,y)\right|^2+1\right) \abs{e_j(y)}^2\left(\int_\R  \left|G(t-v, x, y)-G(s-v, x, y)\right|\,dx\right)\,dy.
\end{multline*}
By Corollary $\ref{cor4}$, we obtain
\[
\int_\R\left|\varsigma_{j,t}(u)(v,x)-\varsigma_{j,s}(u)(v,x)\right|^2 dx
\leq C(s-v)^{-\delta}\,(t-s)^\delta  \!\int_\R\left(\left|u(v,y)\right|^2+1\right) \abs{e_j(y)}^2\! dy.
\]
Consequently, by \eqref{eq:bound-lg},
\[
\left\|\varsigma_{j,t}(u)(v,\cdot)-\varsigma_{j,s}(u)(v,\cdot)\right\|_2\leq \,C\,\,(s-v)^{-\frac{\delta}{2}}\,(t-s)^{\frac{\delta}{2}}\,\left (\left\| u(v,\cdot)\right\|_2 + 1\right).
\]

$(ii)$
For the second statement, we have
\begin{align*}
\MoveEqLeft[1]
\left\| \varsigma_{j,t}(u)(v,\cdot)- \varsigma_{j,s}(u)(v,\cdot)-\varsigma_{j,t}(u)(r,\cdot)+ \varsigma_{j,s}(u)(r,\cdot) \right\|_2^2
\\*
&=\int_\R \left| \varsigma_{j,t}(u)(v,x)- \varsigma_{j,s}(u)(v,x)-\varsigma_{j,t}(u)(r,x)+ \varsigma_{j,s}(u)(r,x)\right|^2\,dx
\\
&=\int_\R \Big|\int_\R \Big[ G(t-v, x, y)-G(s-v,x,y)\Big]\,h(u(v,y))\,e_j(y)\,dy
\\*
&\quad -\int_\R \Big[ G(t-r, x, y)-G(s-r, x, y)\Big]\,h(u(r,y))\,e_j(y)\,dy\Big|^2\,dx
\\
&=\int_\R \Biggl|\int_\R \Big[ G(t-v, x, y)-G(s-v,x,y)-G(t-r, x, y)+G(s-r, x, y)\Big]
\\*
&\quad\times h(u(r,y))\,e_j(y)\,dy
 +\int_\R \Big[ G(t-v, x, y)-G(s-v, x, y)\Big]
 \\
&\quad\times \Big[h(u(v,y))-h(u(r,y))\Big]\,e_j(y)\,dy\Biggr|^2\,dx
\\
&\leq C\int_\R \Biggl|\int_\R \Big[ G(t-v, x, y)-G(s-v,x,y)-G(t-r, x, y)+G(s-r, x, y)\Big]
\\
&\quad\times h(u(r,y))\,e_j(y)\,dy\Biggr|^2\,dx
+C\int_\R\Biggl|\int_\R \Big[ G(t-v, x, y)-G(s-v, x, y)\Big]
\\
&\quad\times\Big[h(u(v,y))-h(u(r,y))\Big]\,e_j(y)\,dy\Biggr|^2\,dx
= K_1+K_2.
\end{align*}
In order to get the upper bound for $K_1,$ we apply  Holder's inequality and produce that
\begin{align*}
K_1&=\int_\R \biggl|\int_\R \Big[ G(t-v, x, y)-G(s-v,x,y)-G(t-r, x, y)+G(s-r, x, y)\Big]
\\*
&\quad\times h(u(r,y))\,e_j(y)\,dy\biggr|^2\,dx
\\
&\leq\,\int_\R \left(\int_\R \Big| G(t-v, x, y)-G(s-v,x,y)-G(t-r, x, y)+G(s-r, x, y)\Big|\,dy\right)
\\
&\quad\times\biggl(\int_\R \abs{ G(t-v, x, y)-G(s-v,x,y)-G(t-r, x, y)+G(s-r, x, y)}
\\
&\quad\quad\times \abs{h(u(r,y))}^2\left|e_j(y)\right|^2  dy\biggr)dx
\\
&\leq\, C \int_\R \left(\int_\R \left| G(t-v, x, y)-G(s-v,x,y)-G(t-r, x, y)+G(s-r, x, y)\right|\,dy\right)
\\
&\quad\times\biggl(\int_\R \abs{G(t-v, x, y)-G(s-v,x,y)-G(t-r, x, y)+G(s-r, x, y)}
\\
&\quad\quad\times(\abs{u(r,y)}^2+1) \abs{e_j(y)}^2 dy\biggr)dx
\end{align*}
By Corollary \ref{cor2},
\begin{multline*}
K_1\leq C
\int_{\R^2}\!\abs{G(t-v, x, y)-G(s-v,x,y)-G(t-r, x, y)+G(s-r, x, y)}
\\
\times(\abs{u(r,y)}^2+1) \abs{e_j(y)}^2dy\,dx
\end{multline*}
 Lemma $\ref{lem6}$ and  Fubini's theorem allow us to get
\begin{align*}
K_1 &\leq\, C
\int_{\R}\left(\left|u(r,y)\right|^2+1\right) \abs{e_j(y)}^2
\\*
&\quad\times\int_\R\left[ G(t-v, x, y)-G(s-v,x,y)-G(t-r, x, y)+G(
s-r,x,y)\right]\,dx\,dy
\\
&\leq\,C\,(t-s)^{\delta'}(s-v)^{-2\delta'}(v-r)^{\delta'} \int_{\R}\left(\left|u(r,y)\right|^2+1\right) \abs{e_j(y)}^2\,dy.
\end{align*}
Therefore, by \eqref{eq:bound-lg},
\[
K_1 \leq\,C\, (t-s)^{\delta'}(s-v)^{-2\delta'}(v-r)^{\delta'}\left(\left\|u(r,\cdot) \right\|^2_2+1\right).
\]
In order to get the upper bound for $K_2,$ we apply again the  H\"older's inequality and get that
\begin{align*}
K_2&=\int_\R\Big|\int_\R \Big[ G(t-v, x, y)-G(s-v, x, y)\Big]\,\Big[h(u(v,y))-h(u(r,y))\Big]\,e_j(y)\,dy\Big|^2\,dx
\\
&\leq\,\int_\R\left(\int_\R \Big| G(t-v, x, y)-G(s-v, x, y)\Big|\,dy \right)
\\
&\;\;\times \left(\int_\R\Big| G(t-v, x, y)-G(s-v, x, y)\Big|\Big[h(u(v,y))-h(u(r,y))\Big]^2\left|e_j(y)\right|^2dy\right)dx
\\
&\leq C \sup_{j\in\mathbb{N}}\left\|e_j\right\|^2_\infty\,\int_\R\left(\int_\R \Big| G(t-v, x, y)-G(s-v, x, y)\Big|\,dy \right)
\\
&\;\;\times
 \left(\int_\R\Big| G(t-v, x, y)-G(s-v, x, y)\Big|\Big|u(v,y)-u(r,y)\Big|^2dy\right)dx.
\end{align*}
By Corollary \ref{cor2},
\[
K_2\leq\,C\,\int_{\R^2}\Big| G(t-v, x, y)-G(s-v, x, y)\Big|\Big|u(v,y)-u(r,y)\Big|^2\,dy\,dx
\]
Therefore, applying Corollary \ref{cor4} and    Fubini's theorem we get that
\begin{align*}
K_2 &\leq\,C\, (t-s)^{\delta}\,(s-v)^{-\delta}\,\int_{\R}\Big|u(v,y)-u(r,y)\Big|^2\,dy
\\
&=\,C\, (t-s)^{\delta}\,(s-v)^{-\delta}\,\left\| u(v,\cdot)-u(r,\cdot) \right\|_2^2
\\
&\leq\,C\, (t-s)^{\delta}\,(s-v)^{-\delta}\,\left\| u(v,\cdot)-u(r,\cdot) \right\|_2^2
\end{align*}
Consequently,
\begin{align*}
\MoveEqLeft
\left\| \varsigma_{j,t}(u)(v,\cdot)- \varsigma_{j,s}(u)(v,\cdot)-\varsigma_{j,t}(u)(r,\cdot)+ \varsigma_{j,s}(u)(r,\cdot) \right\|_2^2
\\*
&\leq C \biggl( (t-s)^{\delta} (s-v)^{-\delta}\,\left\| u(v,\cdot)-u(r,\cdot) \right\|_2^2 
\\*
&\quad+ (t-s)^{\delta'}(s-v)^{-2\delta'}(v-r)^{\delta'}  \left(\left\|u(r,\cdot) \right\|^2_2+1\right ) \biggr) .
\end{align*}
\end{proof}

\begin{lemma}\label{lem9b}
Let  $\{u(t,\cdot),t\in[0,T]\}$ and $\{\tilde u(t,\cdot),t\in[0,T]\}$ be two $L^2(\R)$-valued random processes.
Then  
\begin{enumerate}[(i)]
\item
for all $0<s<t<T$,
\[
\sup_{j\in \mathbb{N}}\left\| \varsigma_{j,t}(u)(s,\cdot)- \varsigma_{j,t}(\tilde{u})(s,\cdot) \right\|_2 \leq \,C\, \left\|u(s,\cdot)-\tilde{u}(s,\cdot)\right\|_2,
\]
\item
for all $0<r<s<t<T$ and for every $\delta\in (\frac13,1)$,
\begin{multline*}
\sup_{j\in \mathbb{N}}\left\| \varsigma_{j,t}(u)(s,\cdot)- \varsigma_{j,t}(\tilde{u})(s,\cdot)-\varsigma_{j,t}(u)(r,\cdot)+ \varsigma_{j,t}(\tilde{u})(r,\cdot) \right\|_2
\\
\leq C (t-s)^{-\delta}\,(s-r)^{\delta}\,\left\| u(r,\cdot)-\tilde{u}(r,\cdot)\right\|_2+\,C \left\|u(s,\cdot)-\tilde{u}(s,\cdot)-u(r,\cdot)+\tilde{u}(r,\cdot)   \right\|_2.
\end{multline*}
\end{enumerate}
\end{lemma}
\begin{proof}
In order to prove $(i)$, we can apply the Cauchy--Schwartz  inequality and Corollary \ref{cor2}, and get the following inequalities:
\begin{align*}
\MoveEqLeft[0]
\left\| \varsigma_{j,t}(u)(s,\cdot)- \varsigma_{j,t}(\tilde{u})(s,\cdot) \right\|_2^2:= \int_\R\left|   \varsigma_{j,t}(u)(s,x)- \varsigma_{j,t}(\tilde{u})(s,x)  \right|^2\,dx
\\
&= \int_\R\left| \int_\R G(t-s, x, y)\left[h(u(s,y))-h(\tilde{u}(s,y)) \right]e_j(y)\,dy\right|^2dx
\\
&\leq C\int_\R \left(\int_\R |G(t-s, x, y)|dy \right)
  \left(\int_\R |G(t-s, x, y)|\left[ h(u(s,y))-h(\tilde{u}(s,y)) \right]^2dy \right) dx
	\\ 
&\leq C \int_\R\left[ h(u(s,y))-h(\tilde{u}(s,y)) \right]^2\left(\int_\R |G(t-s, x, y)|dx \right) dy
	\\
&\leq C\left\|u(s,\cdot)-\tilde{u}(s,\cdot)\right\|_2^2,
\end{align*}
whence $(i)$ follows.
To proceed with $(ii)$, we apply the Minkowski inequality in order to get the following bounds:
\begin{align*}
\MoveEqLeft[0]
\left\| \varsigma_{j,t}(u)(s,\cdot)- \varsigma_{j,t}(\tilde{u})(s,\cdot)-\varsigma_{j,t}(u)(r,\cdot)+ \varsigma_{j,t}(\tilde{u})(r,\cdot) \right\|_2^2
\\
&= \int_\R\Big|\int_\R \left[ G(t-s, x, y)- G(t-r, x, y) \right]\left[ h(u(r,y))-h(\tilde{u}(r,y))\right]e_j(y)
\\
&\quad+ G(t-s, x, y)\left[ h(u(s,y))-h(\tilde{u}(s,y)) -h(u(r,y))+h(\tilde{u}(r,y))  \right]e_j(y)\,dy   \Big|^2\,dx
\\
&\leq C \int_\R\biggl|\int_\R \left| G(t-s, x, y)- G(t-r, x, y) \right|\left| h(u(r,y))-h(\tilde{u}(r,y))\right|dy\biggr|^2dx
\\
&\quad+ C \int_\R\biggl|\int_\R |G(t-s, x, y)|\left| h(u(s,y))-h(\tilde{u}(s,y)) -h(u(r,y))+h(\tilde{u}(r,y))  \right|dy   \biggr|^2\!dx
\\
&\eqqcolon H_1+H_2.
\end{align*}
It follows from the Cauchy--Schwartz inequality that
\begin{multline*}
H_1\leq C\int_\R \left( \int_\R\left| G(t-s, x, y)- G(t-r, x, y) \right|\,dy \right)
\\*
\times\left( \int_\R\left| G(t-s, x, y)- G(t-r, x, y) \right|\, \left[ h(u(r,y))-h(\tilde{u}(r,y))\right]^2  dy \right)\,dx.
\end{multline*}
By  Fubini's theorem,
\begin{align*}
H_1 &\leq  C\int_\R(t-s)^{-\delta}\,(s-r)^{\delta}
\\
&\quad\times\left( \int_\R\left| G(t-s, x, y)- G(t-r, x, y) \right|\, \left[ h(u(r,y))-h(\tilde{u}(r,y))\right]^2  dy \right)\,dx
\\ 
&\leq  C(t-s)^{-\delta}\,(s-r)^{\delta}
\\
&\quad\times\int_\R\left( \left[ h(u(r,y))-h(\tilde{u}(r,y))\right]^2\int_\R\left| G(t-s, x, y)- G(t-r, x, y) \right|\,   dx \right)\,dy.
\end{align*}
Then, using  Corollary \ref{cor4} we obtain
\[
H_1 \leq \,C\, (t-s)^{-2\delta}\,(s-r)^{2\delta}\,\left\| u(r,\cdot)-\tilde{u}(r,\cdot)\right\|_2^2.
\]
For the other term, we get by using Cauchy-Schwartz inequality, Lemma \ref{lem1},  and Fubini's theorem,
\begin{align*}
H_2&\leq\,C\int_\R\left(\int_\R |G(t-s, x, y)|\,dy\right)
\\
&\quad\times
\left(\int_\R |G(t-s, x, y)|\left[ h(u(s,y))-h(\tilde{u}(s,y)) -h(u(r,y))+h(\tilde{u}(r,y))  \right]^2 dy\right)dx
\\
&\leq C \int_\R
\biggl(\left[ h(u(s,y))-h(\tilde{u}(s,y)) -h(u(r,y))+h(\tilde{u}(r,y))  \right]^2
\\
&\quad\times
\int_\R |G(t-s, x, y)|dx\biggr)dy,
\end{align*}
and Lemma \ref{lem1} allow us to get
\[
H_2 \leq C \left\|h(u(s,\cdot))-h(\tilde{u}(s,\cdot))-h(u(r,\cdot))+h(\tilde{u}(r,\cdot))   \right\|_2^2.
\]
It follows from $(\ref{lin})$ that 
\[
H_2 \leq C \left\|u(s,\cdot)-\tilde{u}(s,\cdot)-u(r,\cdot) +\tilde{u}(r,\cdot)\right\|_2^2.
\]
Hence, we get the result.
\end{proof} 

\begin{lemma}\label{lem10}
Let  $\{u(t,\cdot),t\in[0,T]\}$ and $\{\tilde u(t,\cdot),t\in[0,T]\}$ be two $L^2(\R)$-valued random processes.
Then 
\begin{enumerate}[(i)]
\item
 for all    $0<v<s<t<T$  and  for every $\delta\in (\frac13,1)$,
\[
\sup_{j\in \mathbb{N}}\left\| \varsigma^*_{j,t,s}(u)(v,\cdot)- \varsigma^*_{j,t,s}(\tilde{u})(v,\cdot) \right\|_2 \leq\,C \,(t-s)^{\frac{\delta}{2}}\,(s-v)^{-\frac{\delta}{2}}\, \left\|u(v,\cdot)-\tilde{u}(v,\cdot)\right\|_2,
\] 
\item
 for all    $0<r<v<s<t<T$  and for any $\delta \in(\frac{1}{5},1),$ $\delta' \in(\frac{1}{3},1)$,
\begin{align*}
\MoveEqLeft\sup_{j\in \mathbb{N}}\left\| \varsigma^*_{j,t,s}(u)(v,\cdot)- \varsigma^*_{j,t,s}(\tilde{u})(v,\cdot)-\varsigma^*_{j,t,s}(u)(r,\cdot)+ \varsigma^*_{j,t,s}(\tilde{u})(r,\cdot) \right\|_2
\\
&\leq C (t-s)^{\frac{\delta}{2}}\,(s-v)^{-\delta}\,(v-r)^{\frac{\delta}{2}}\,\left\| u(r,\cdot)-\tilde{u}(r,\cdot)\right\|_2
\\
&\quad + C (t-s)^{\frac{\delta'}{2}}\,(s-v)^{-\frac{\delta'}{2}} \left\|u(s,\cdot)-\tilde{u}(s,\cdot)-u(r,\cdot)+\tilde{u}(r,\cdot)   \right\|_2.
\end{align*}
\end{enumerate}
\end{lemma}
\begin{proof}
In order to prove $(i)$, we can apply the Cauchy--Schwartz  inequality, Corollaries \ref{cor2} and \ref{cor4}, and get the following inequalities:
\begin{align*}
\MoveEqLeft
\left\| \varsigma^*_{j,t,s}(u)(v,\cdot)- \varsigma^*_{j,t,s}(\tilde{u})(v,\cdot) \right\|_2^2
= \int_\R\left|   \varsigma^*_{j,t,s}(u)(v,x)- \varsigma^*_{j,t,s}(\tilde{u})(v,x)  \right|^2\,dx
\\
&= \int_\R\left| \int_\R \left[G(t-v, x, y)-G(s-v, x, y)\right]\left[h(u(v,y))-h(\tilde{u}(v,y)) \right]e_j(y)\,dy\right|^2dx
\\
&\leq C \int_\R \left(\int_\R \left|G(t-v, x, y)-G(s-v, x, y)\right|dy \right)
\\
&\quad\times\left(\int_\R \left|G(t-v, x, y)-G(s-v, x, y)\right|\left[ h(u(v,y))-h(\tilde{u}(v,y)) \right]^2dy \right) dx
\\
&\leq C \int_\R\left[ h(u(s,y))-h(\tilde{u}(s,y)) \right]^2\left(\int_\R | G(t-v, x, y)-G(s-v, x, y) |dx \right) dy
\\
&\leq C (t-s)^{\delta}\,(s-v)^{-\delta}  \;\left\|u(s,\cdot)-\tilde{u}(s,\cdot)\right\|_2^2.
\end{align*}
In the other hand, we have by using the fact that for every $a,b\in \R,\;(a+b)^2 \leq 2(a^2+b^2)$
\begin{align*}
\MoveEqLeft
\left\| \varsigma^*_{j,t,s}(u)(v,\cdot)- \varsigma^*_{j,t,s}(\tilde{u})(v,\cdot)-\varsigma^*_{j,t,s}(u)(r,\cdot)+ \varsigma^*_{j,t,s}(\tilde{u})(r,\cdot) \right\|_2^2
\\
&= \int_\R\left| \varsigma^*_{j,t,s}(u)(v,x)- \varsigma^*_{j,t,s}(\tilde{u})(v,x)-\varsigma^*_{j,t,s}(u)(r,x)+ \varsigma^*_{j,t,s}(\tilde{u})(r,x) \right|^2\,dx
\\ 
&= \int_\R \biggl| \int_\R [G(t-v, x, y)-G(s-v, x, y)]
\\
&\quad\times[ h(u(v,y))-h(\tilde{u}(v,y))-h(u(r,y))+h(\tilde{u}(r,y)) ] e_j(y)\,dy 
\\
&+ \int_\R [h(u(r,y))-h(\tilde{u}(r,y))]
\\
&\quad\times[G(t-v, x, y)-G(s-v, x, y)-G(t-r, x, y)+G(s-r, x, y)]e_j(y)\,dy\biggr|^2\!dx
\\ 
&\leq C\int_\R \biggl| \int_\R [G(t-v, x, y)-G(s-v, x, y)]
\\*
&\quad\times[h(u(v,y))-h(\tilde{u}(v,y))-h(u(r,y))+h(\tilde{u}(r,y)) ] e_j(y)\,dy\biggr|^2 dx  
\\
&\quad + C\int_\R\biggl|\int_\R [h(u(r,y))-h(\tilde{u}(r,y))]
\\
&\quad\times[G(t-v, x, y)-G(s-v, x, y)-G(t-r, x, y)+G(s-r, x, y)]e_j(y)\,dy\biggr|^2\! dx
\\ 
&\eqqcolon D_1+D_2.
\end{align*}
Applying the Cauchy--Schwarz inequality and Fubini's theorem, combined with Corollaries \ref{cor2} and \ref{cor4}, we get 
\begin{align*}
D_1 &= \int_\R \Big| \int_\R [G(t-v, x, y)-G(s-v, x, y)]
\\
&\quad\times[ h(u(v,y))-h(\tilde{u}(v,y))-h(u(r,y))+h(\tilde{u}(r,y)) ] e_j(y)\,dy\Big|^2 dx 
\\
&\leq \int_\R\left(\int_\R | G(t-v, x, y)-G(s-v, x, y)|dy\right)
\biggl( \int_\R |G(t-v, x, y)-G(s-v, x, y)|
\\
&\quad\times| h(u(v,y))-h(\tilde{u}(v,y))-h(u(r,y))+h(\tilde{u}(r,y)) |^2 \,dy\biggr)  dx 
\\
&\leq  C\, (t-s)^{\delta'}\,(s-v)^{-\delta'} \left\|u(s,\cdot)-\tilde{u}(s,\cdot)-u(r,\cdot)+\tilde{u}(r,\cdot)   \right\|_2^2.
\end{align*} 
By the same technique as before and from Lemma \ref{lem6}, we obtain
\begin{align*}
D_2 &=\int_\R \biggl| \int_\R [G(t-v, x, y)-G(s-v, x, y)-G(t-r, x, y)+G(s-r, x, y)]
\\*
&\quad\times[ h(u(r,y))-h(\tilde{u}(r,y)) ] e_j(y)\,dy\biggr|^2 dx 
\\
&\leq \int_\R\left(\int_\R | G(t-v, x, y)-G(s-v, x, y)-G(t-r, x, y)+G(s-r, x, y)|dy\right)
\\
&\quad\times  \biggl( \int_\R |G(t-v, x, y)-G(s-v, x, y)-G(t-r, x, y)+G(s-r, x, y)|
\\
&\quad\quad\times| h(u(r,y))-h(\tilde{u}(r,y)) |^2 \,dy\biggr)  dx 
\\
&\leq  C\, (t-s)^{\delta}\,(s-v)^{-2\delta} (v-r)^{\delta}\left\|u(r,\cdot)-\tilde{u}(r,\cdot)   \right\|_2^2.
\end{align*}
Hence, we get the result. 
\end{proof}

\section*{Acknowledgments}
Y. Mishura and K. Ralchenko acknowledge that the present research is carried through within the
frame and support of the ToppForsk project nr. 274410 of the Research Council of
Norway with title STORM: Stochastics for Time-Space Risk Models.

\end{document}